\documentclass[12pt]{amsart}
\usepackage{amsthm, amsmath, amssymb, mathrsfs, enumitem, svn, amsfonts, stmaryrd, mathtools}
\usepackage{verbatim, comment} 

\usepackage{mathscinet}
\usepackage{tikz-cd}

\usepackage[a4paper,bindingoffset=0.2in,
            left=1in,right=1in,
            top=0.4in,bottom=0.35in,
            footskip=.25in]{geometry}

\usepackage{tabu} 
\usepackage[all, cmtip]{xy}

\usepackage{color}
\definecolor{ghcolor}{RGB}{0, 150, 200} 
\definecolor{winestain}{rgb}{0.5,0,0}
\usepackage[linktocpage=true, hidelinks, colorlinks, linkcolor=blue, citecolor=red, bookmarksopen=true]{hyperref} 

\newtheorem{theorem}[subsubsection]{Theorem}
\newtheorem{thm}[subsubsection]{Theorem}
\newtheorem{lemma}[subsubsection]{Lemma}

\newtheorem{cor}[subsubsection]{Corollary}

\newtheorem{prop}[subsubsection]{Proposition}

\theoremstyle{definition}
\newtheorem{defn}[subsubsection]{Definition}
\newtheorem{construction}[subsubsection]{Construction}

\newtheorem{assumption}[subsubsection]{Assumption}

\newtheorem{notation}[subsubsection]{Notation}

\newtheorem{convention}[subsubsection]{Convention}

\theoremstyle{remark}

\newtheorem{rem}[subsubsection]{Remark}
\numberwithin{equation}{subsection}
\def \into {\hookrightarrow }
\def \to {\rightarrow}
\def \onto {\twoheadrightarrow}
\renewcommand{\projlim}{\varprojlim}

\newcommand{\GL}{\mathrm{GL}}

\DeclareMathOperator{\Mod}{Mod}

\DeclareMathOperator{\Rep}{Rep}

\DeclareMathOperator{\Hom}{Hom}

\DeclareMathOperator{\Ker}{Ker}

\DeclareMathOperator{\Gal}{Gal}

\DeclareMathOperator{\Fil}{Fil}

\DeclareMathOperator{\Spec}{Spec}

\def\et{\mathrm{\acute{e}t}}

\newcommand{\ur}{\mathrm{ur}}

\newcommand{\Kinfty}{{K_{\infty}}}

\def\cris{{\mathrm{cris}}}

\def\st{{\mathrm{st}}}

\def\dR{{\mathrm{dR}}}
\def\HT{{\mathrm{HT}}}

\newcommand{\Ainf}{\mathbf{A}_{\mathrm{inf}}}

\newcommand{\acris}{\mathbf{A}_{\mathrm{cris}}}

\newcommand{\bcris}{\mathbf{B}_{\mathrm{cris}}}

\newcommand{\bst}{\mathbf{B}_{\mathrm{st}}}

\newcommand{\bdr}{\mathbf{B}_{\mathrm{dR}}}

\newcommand{\bht}{\mathbf{B}_{\mathrm{HT}}}

\newcommand{\BT}{\mathrm{BT}}

\newcommand{\bbcris}{\mathbb{B}_{\mathrm{cris}}}
\newcommand{\bbst}{\mathbb{B}_{\mathrm{st}}}
\newcommand{\bbdr}{\mathbb{B}_{\mathrm{dR}}}
\newcommand{\bbht}{\mathbb{B}_{\mathrm{HT}}}

\newcommand{\wta}{   {\widetilde{{\mathbf{A}}}}  }
\newcommand{\wtb}{   {\widetilde{{\mathbf{B}}}}  }

\newcommand{\wtE}{   {\widetilde{{\mathbf{E}}}}  }

 \def \E   {{\mathcal E}}
 \def \O {{\mathcal{O}}}
  
 \renewcommand{\OE}{{\mathcal{O}_{\mathcal E}}}

\def \OK {{\mathcal{O}_K}}
\def \ok {{\mathcal{O}_K}}
\def \ko {{K_{0}}}
\def \oc {{\mathcal{O}_C}}

\def \oko {{\mathcal{O}_{K_0}}}

\newcommand*{\wt}[1]{\widetilde{#1}}
\newcommand*{\wh}[1]{\widehat{#1}}

\newcommand{\Zp}{{\mathbb{Z}_p}}
\newcommand{\Qp}{{\mathbb{Q}_p}}
\newcommand{\Fp}{{\mathbb{F}_p}}
\newcommand{\zp}{{\mathbb{Z}_p}}
\newcommand{\qp}{{\mathbb{Q}_p}}

\newcommand{\cm}{\mathcal{M}} 

\def \cn {\mathcal N}

\newcommand{\gs}{{\mathfrak{S}}}

\newcommand{\gu}{{\mathfrak{u}}}
\newcommand{\gm}{{\mathfrak{M}}}

\newcommand{\gn}{{\mathfrak{N}}}

\def \hM {{\widehat{\mathfrak{M}}} }
\def \hm {{\widehat{\mathfrak{M}}} }

\newcommand{\bfD}{ {\mathbf{D}}}

\newcommand{\bfm}{ {\mathbf{M}}}

 \newcommand{\bbk}{{\mathbb{K}}}
  \newcommand{\bbd}{{\mathbb{D}}}
   
 \newcommand{\bbK}{{\mathbb{K}}}
\newcommand{\bbko}{{\mathbb{K}_0}}
\newcommand{\bbkinfty}{{\bbk_\infty}}

\newcommand{\ZZ}{\mathbb{Z}}

\renewcommand{\phi}{\varphi}

\usepackage{relsize}
\usepackage[bbgreekl]{mathbbol}
\usepackage{amsfonts}

\DeclareSymbolFontAlphabet{\mathbb}{AMSb}
\DeclareSymbolFontAlphabet{\mathbbl}{bbold}

\newcommand{\MFnabla}{\mathbf{MF}_{K/K_0}(\varphi, N, \nabla)} 
\newcommand{\MFnablageq}{\mathbf{MF}_{K/K_0}^{\geq 0}(\varphi, N, \nabla)}
\newcommand{\MFkpf}{\mathbf{MF}_{\bbk/\bbk_0}(\varphi, N)}
\newcommand{\MFkpfgeq}{\mathbf{MF}_{\bbk/\bbk_0}^{\geq 0}(\varphi, N )}
\newcommand{\MFnablawa}{\mathbf{MF}^{\mathrm{wa}}_{K/K_0}(\varphi, N, \nabla)}
\newcommand{\MFnablawageq}{\mathbf{MF}^{\geq 0, \mathrm{wa}}_{K/K_0}  (\varphi, N, \nabla)}
 \newcommand{\MFa}{\mathbf{MF}^{\mathrm{a}}_{K/K_0}(\varphi, N, \nabla)}
 \newcommand{\Modgsnabla}{\Mod_{\gs}(\varphi, N, \nabla)}

\newcommand{\ModA}{\Mod_{\mathcal A}(\varphi, N_\nabla, \nabla)}
\newcommand{\ModAO}{\Mod^{0}_{\mathcal A}(\varphi, N_\nabla, \nabla)}

\begin{document}

\title[]{Integral $p$-adic Hodge theory in the imperfect residue field case}

\date{\today}

\author[]{Hui Gao} 
\address{Department of Mathematics, Southern University of Science and Technology, Shenzhen 518055,  China}
\email{gaoh@sustech.edu.cn}

\begin{abstract}
Let $K$ be a mixed characteristic complete discrete valuation field with residue field admitting a finite $p$-basis, and let $G_K$ be the Galois group. 
We first classify  semi-stable representations of $G_K$ by {weakly admissible filtered $(\varphi,N)$-modules with connections}. We then construct a fully faithful functor from the category of \emph{integral} semi-stable representations of $G_K$ to the category of  Breuil-Kisin  $ G_K$-modules. Using the integral theory, we classify $p$-divisible groups over the ring of integers of $K$ by minuscule  Breuil-Kisin modules with connections.
\end{abstract}

\maketitle
\tableofcontents

\section{Introduction}
\subsection{Overview and main theorems}

Let $p$ be a prime.
In this paper, we study $p$-adic Hodge theory, where we use various (semi-)linear algebra data to classify $p$-adic Galois representations. 
The theory works particularly well if one considers the Galois group  of a CDVF (complete discrete valuation field)  of mixed characteristic $(0, p)$ with \emph{perfect} residue field such as a finite extension of $\Qp$.


To study $p$-adic local systems of a general rigid space --which arise naturally, e.g., in the study of Shimura varieties -- one is lead to the theory of \emph{relative $p$-adic Hodge theory}. In this context, mixed characteristic CDVFs    with \emph{imperfect} residue fields naturally arise. For example, let $R=\Zp\langle T^{\pm 1}\rangle$ be the (integral) ring corresponding to a rigid torus, and let $\wh{R_{(p)}}$ be the $p$-adic completion of the localization $R_{(p)}$, then $\wh{R_{(p)}}[1/p]$ is a CDVF with residue field $\Fp(T)$. 
Indeed, localizing $R$ at some other ideals also give rise to CDVFs   with  imperfect  residue fields, cf. e.g. \cite[\S 3.3]{Bri08}. To understand  $p$-adic representations of $\pi_1^{\et}(R[1/p])$) (called \emph{``the relative case"} in the following), it turns out to be crucial to  understand representations  for these CDVFs   with  imperfect  residue fields  (called \emph{``the imperfect residue field case"} in the following), such as  representations of the Galois group of ${\wh{R_{(p)}}[1/p]}$.

In fact, Brinon studies the Hodge-Tate, de Rham, and crystalline representations in the imperfect residue field case  in \cite{Bri06}, which then pave the way for the studies in the relative case in \cite{Bri08}. 
Throughout this paper, let $K$ be a CDVF of characteristic $0$ with residue field $k_K$ of characteristic $p$ such that $[k_K:k_K^p]< \infty$, and let $G_K$ be the Galois group. 
Our first main result is the following.

\begin{theorem}\label{thm1} (= Thm. \ref{thmCF})
 The category of  semi-stable representations of $G_K$ is equivalent to the category of  \emph{weakly admissible filtered $(\varphi,N)$-modules with connections} (cf. Def. \ref{deffilnmod}).
\end{theorem}

\begin{rem}
\begin{enumerate}
\item  When the residue field of $K$ is \emph{perfect}, then this is the classical theorem of Colmez-Fontaine \cite{CF00}; note that in the perfect residue field case, the connection is automatically a zero map.
\item For general $K$, the crystalline case of Thm. \ref{thm1} is obtained in \cite{Bri06}.
\end{enumerate}
\end{rem}

We now have quite good understanding of \emph{general} $p$-adic Galois representations in the relative case; for example, the ``relative $(\varphi, \Gamma)$-modules" (as well as their overconvergent versions) are constructed in \cite{And06, AB08, KL15, KL2}. 
Next thing in order is to study the Hodge-Tate, de Rham, crystalline and semi-stable representations in the relative case. Brinon initiated these studies in the Hodge-Tate, de Rham, and crystalline case in \cite{Bri08}; indeed, with Thm. \ref{thm1} established, we should be able to obtain some similar results in the semi-stable case. 
Recently, the rigidity theorem of Liu-Zhu \cite{LZ17} provides a surprisingly simple description of de Rham local systems. Then, a rigidity theorem for \emph{almost} Hodge-Tate local systems is established by Shimizu \cite{Shi18}. Further, Shimizu obtains a variant of $p$-adic local monodromy theorem in the relative case in \cite{Shipst}. In summary, it seems that we now have some quite interesting understanding of local systems in the de Rham (and almost Hodge-Tate) case, but perhaps not so much in the crystalline or semi-stable case; for example, our understanding of the notion ``weak admissibility" is not completely satisfactory so far, cf. the discussion on \cite[p. 136]{Bri08}.


In the study of crystalline and semi-stable representations in the \emph{perfect} residue field case, an important phenomenon is that we can give very precise classifications of \emph{integral} crystalline and semi-stable representations, the study of which we call   \emph{integral $p$-adic Hodge theory}.
The second main theorem in the paper is the following Thm. \ref{thm2}, which we regard  as a very first step towards the study of \emph{relative integral  $p$-adic Hodge theory}.  

\begin{theorem}\label{thm2} (= Thm. \ref{thm411})
There is a fully faithful functor from the category of \emph{integral} semi-stable representations of $G_K$ to the category of    \emph{Breuil-Kisin  $ G_K$-modules} (cf. Def. \ref{defwr}).
\end{theorem}

\begin{rem}
\begin{enumerate}
\item When the residue field of $K$ is \emph{perfect}, then the main result of \cite{Gaolp}
says that the functor in the theorem is indeed an equivalence.

\item In comparison with Thm. \ref{thm1}, it might seem tempting to construct some category of Breuil-Kisin  $ G_K$-modules \emph{with connections}. However, it seems difficult to characterize the properties of the possible connection operator, cf. Rem. \ref{rem437}. Nonetheless, if we focus on integral crystalline representations with Hodge-Tate weights in $\{0, 1\}$, which correspond to $p$-divisible groups over $\ok$ (the ring of integers of $K$), then we can indeed construct some category with connection operators, cf. Thm. \ref{thm3} in the following. 
\end{enumerate}
\end{rem}


Using the integral theory developed in this paper, we can classify $p$-divisible groups over $\ok$.

\begin{theorem}\label{thm3} (= Thm. \ref{68})
The category of $p$-divisible groups over $\ok$ is equivalent to the category of minuscule Breuil-Kisin modules with connections.
\end{theorem}

\begin{rem}\label{r116}
\begin{enumerate}
\item When $K$ has \emph{perfect} residue field, the theorem is due to Kisin \cite{Kis06} when $p>2$, and independently to  Kim \cite{Kim12}, Lau \cite{Lau14} and Liu \cite{Liu13} when $p=2$.
\item For general $K$, the theorem is   known when $p>2$ (cf. Rem. \ref{re633}): one can deduce it by an easy combination of results in \cite{BT08} and a theorem of Caruso-Liu \cite{CL09}.
(Alternatively, it can also be regarded as a special case  of results of Kim \cite{Kim15}, where  $p$-divisible groups for $p>2$ in the \emph{relative} case are classified).
\item When $p=2$ (and for general $K$), new ideas are needed. In fact, a key ingredient in the proof is that we can first show that the category of $p$-divisible groups over $\ok$ is equivalent to the category of integral crystalline representations of $G_K$ with Hodge-Tate weights in $\{0, 1\}$. 
Then we use our integral theory  to construct the equivalence with the category of minuscule Breuil-Kisin modules  with connections. 
\item As a consequence of  Thm. \ref{thm3}, we can classify finite flat group schemes over $\ok$, cf. Thm. \ref{fflat}.
\end{enumerate}
\end{rem}

We propose some speculations and questions for future investigations.
\begin{rem}
\begin{enumerate}
\item What is the possible generalization of Thm. \ref{thm2} in the \emph{relative} case?

\item What is the relation between Thm. \ref{thm3} and the \emph{filtered prismatic Dieudonn\'e modules} in \cite{ALB}?

\item As mentioned in Rem. \ref{r116}(2), $p$-divisible groups in the \emph{relative} case are classified by Kim \cite{Kim15} when $p>2$; to study the $p=2$ case, it seems likely the ideas (and possibly even the results) in this paper could be useful.
\end{enumerate}
\end{rem}

\subsection{Structure of the paper}
In \S \ref{sec2}, we use Fontaine modules (i.e., {weakly admissible} filtered $(\varphi,N)$-modules with connections) to classify semi-stable representations.
In \S \ref{sec3}, we build the link from  Fontaine modules to Kisin's $\mathcal A$-modules and $\gs$-modules.
In \S \ref{sec4}, we use {Breuil-Kisin  $ G_K$-modules} to study \emph{integral} semi-stable representations.
In \S \ref{sec5}, we review results about $\varphi$-modules over $\gs$ and $S$.
In \S \ref{secpdiv},  we classify $p$-divisible groups over $\ok$.

\subsection{Some notations and conventions}
\begin{convention}\label{subsubco}
\emph{Categories and co-variant functors.} 
\begin{itemize}
\item In this paper we will define many categories of modules (with various structures); we will always omit the definition of morphisms for these categories, which are always obvious (i.e., module homomorphisms compatible with various structures). 
\item  When we define functors relating various categories, we will always use \emph{co-variant} functors. This makes the comparisons amongst them easier (i.e., using tensor products, rather than $\Hom$'s).
\end{itemize}
\end{convention}

\begin{convention}
\emph{Hodge-Tate weights, and Breuil-Kisin heights.}
\begin{itemize}
\item In accordance with Convention \ref{subsubco}, our $D_\st(V)$ is defined as $(V\otimes_{\Qp} \bst)^{G_K}$, and hence the Hodge-Tate weight of the cyclotomic character $\chi_p$ is $-1$.
\item  Once we move on to study  \emph{integral} theory from \S \ref{sec3}, we will focus on representations with \emph{non-negative} Hodge-Tate weights and Breuil-Kisin modules with \emph{non-negative} $E(u)$-heights. For example, the Breuil-Kisin module associated to $\chi_p^{-1}$  has $E(u)$-height $1$.
\end{itemize}
\end{convention}

\begin{notation} 
Let $H$ be a profinite group, then we use $\Rep_{\Fp}(H)$ (resp. $\Rep_{\Zp}(H)$, resp. $\Rep_{\Qp}(H)$) denote  the category of finite dimensional $\Fp$-vector spaces (resp. finite free $\Zp$-modules, resp. finite dimensional $\Qp$-vector spaces) $V$   with  continuous $\Fp$- (resp. $\Zp$-, resp. $\Qp$-) linear $H$-actions.
Sometimes we put superscripts such as ``$\cris, \st, \geq 0$" etc. over $\Rep_{\Zp}(H)$ or $\Rep_{\Qp}(H)$ for the obvious meaning: i.e., crystalline representations, resp. semi-stable representations, resp. those with Hodge-Tate weights $\geq 0$.
\end{notation}
 
\begin{notation}
Throughout this paper, we reserve $\varphi$ to denote Frobenius operator. We sometimes add subscripts to indicate on which object Frobenius is defined. For example, $\varphi_\gm$ is the Frobenius defined on $\gm$. We always drop these subscripts if no confusion arises.  Given a homomorphism of rings $\varphi: A \to A$ and given an $A$-module $M$, denote $\varphi^{\ast}M: =A\otimes_{\varphi, A} M$. 
For a category $\mathscr C$ whose morphism sets are $\Zp$-modules, we use $\mathscr C \otimes_\Zp \Qp$ to denote its isogeny category.
\end{notation}


\subsection*{Acknowledgement} I thank Olivier Brinon, Heng Du, Tong Liu, Yong Suk Moon, Kazuma Morita and Weizhe Zheng for some useful discussions and correspondences.

\section{Fontaine modules and semi-stable representations}\label{sec2}
 In this section, we first review the Fontaine rings and the Fontaine modules (in the imperfect residue field case). We then show that \emph{weakly admissible} filtered $(\varphi,N)$-modules with connections are \emph{admissible}, and hence classify semi-stable representations.
 
 \subsection{Notations of base fields}\label{sec21}
Let $K$ be a complete discrete valuation field of characteristic $0$ with residue field $k_K$ of characteristic $p>0$ such that $[k_K:k_K^p]=p^d$ where $d\geq 0$.  
Fix an algebraic closure $\overline{K}$ of $K$ and let $G_{K}=\Gal(\overline{K}/K)$.
We first set up  notations for some other fields related with $K$.
 
  The notations below depend on various choices, particularly that of $\ko$ and $\varphi$. The way these choices are made are slightly different in the references \cite{Bri06}, \cite{BT08}, \cite{MoritadR, Moritacrys} and \cite{Ohk13}. But they are all ``equivalent" definitions, cf. the detailed discussions in \cite[1A, 1G]{Ohk13}.  Our exposition here largely follows that in \cite{Bri06}.

\emph{Fix} a closed sub-field $K_0 \subset K$ with residue field $k_K$ and with $p$ as a uniformizer. 
Note that such $K_0$ is not unique (unless $d=0$ or $p$ is a uniformizer of $K$); but the ramification index $e_K=[K: K_0]$ is uniquely determined. Let $\O_K, \O_{K_0}$ be the ring of integers.
Let $\bar{t}_1, \cdots, \bar{t}_d$ be a basis of $k_K$ over $k_K^p$, and \emph{fix} some lifts $t_1, \cdots, t_d \in \O_{K_0}$. \emph{Fix} a Frobenius $\varphi$ on $\oko$ which lifts the $p$-power map on $k_K$; such $\varphi$ is not unique.

Let $\overline{k_K}$ be the residue field of $\overline{K}$, and let $\mathbf{k}$ be the radical closure of $k_K \subset \overline{k_K}$. The choices $K_0$ and $\varphi$ determines a $\varphi$-equivariant embedding 
$$i_\varphi: \oko \into W(\mathbf{k}),$$
where $ W(\mathbf{k})$ is the ring of Witt vectors with $\varphi_{W(\mathbf{k})}$ the usual Frobenius map.
Let $\bbk_0 =W(\mathbf{k})[1/p]$, and fix $\overline{\bbk}$ an algebraic closure of $\bbk_0$. 
The map $i_\varphi$ induces 
$$i_\varphi: \overline{K} \into \overline{\bbk}$$
 with dense image and hence further induces 
 $$i_\varphi: C  \simeq \mathbb{C},$$
  where ${C}$ (resp. $\mathbb{C}$) is the $p$-adic completion of $\overline{K}$ (resp. $\overline{\bbk}$).
Let 
$$\bbk=\bbk_0 i_\varphi(K) \subset \overline{\bbk},$$ and let $G_{\bbK}=\Gal(\overline{\bbK}/\bbK)$. Via $i_\varphi: \overline{K} \into \overline{\bbk}$, we can identify  $G_\bbk$ as a subgroup of $G_K$.


\subsection{Fontaine rings}
In this subsection, we give a quick review the Fontaine rings in the  imperfect residue field case; the presentations   here are not necessarily in the order they were first developed. We refer the readers to \cite{Bri06} and \cite{Moritacrys} for more details.


\begin{notation}
Starting from $\bbk$ --a CDVF with \emph{perfect} residue field--  and the perfectoid field $\mathbb{C}$, we use  
\begin{equation} \label{perfbdr}
\bbdr^+,\quad \bbdr,\quad \bbht,\quad \mathbb{A}_\cris,\quad \bbcris^+,\quad \bbcris,\quad \mathbb{A}_\st,\quad \bbst^+,\quad \bbst
\end{equation}
  to denote the ``usual" Fontaine  rings (in the perfect residue field case).
\end{notation}

\begin{notation} \label{no222}
Indeed, all the rings in \eqref{perfbdr} depend \emph{only} on the perfectoid field $\mathbb{C}$ (and not $\bbk$). Hence starting from the perfectoid field $C$, we can also define the corresponding rings; denote them as 
\begin{equation}
\bdr^{\nabla +},\quad \bdr^{\nabla},\quad \bht^{\nabla},\quad \acris^{\nabla},\quad \bcris^{\nabla +},\quad \bcris^{\nabla},\quad \mathbf{A}_\st^{\nabla},\quad \bst^{\nabla +},\quad \bst^{\nabla}.
\end{equation}
(The superscripts $\nabla$ will be explained in \S \ref{drnabla}.)
The isomorphism $i_\varphi: C  \simeq \mathbb{C}$ induces ring isomorphisms 
$$i_\varphi: \mathbf{B}_\ast^{\nabla} \simeq \mathbb{B}_\ast, \text{ for }  \ast \in \{\dR, \HT, \cris, \st  \}.$$
These isomorphisms are compatible with various $\varphi$-, $N$- and filtration structures, and are $G_\bbk$-equivariant (as we mentioned in \S \ref{sec21}, we regard $G_\bbk$ as a subgroup of $G_K$).
\end{notation}

\begin{notation}
 Since $\O_C$ is a perfectoid ring, we can define its tilt $\O_C^\flat$; let $W(\O_C^\flat)$ be the ring of Witt vectors. 
 Elements in $\O_C^\flat$ are in bijection with sequences $(x^{(n)})_{n\geq 0}$ where $x^{(n)} \in \O_C$ and $(x^{(n+1)})^p=x^{(n)}$. For each $1\leq i \leq d$, fix some $\tilde{t_i} \in \O_C^\flat$ with $\tilde{t_i}^{(0)}=t_i$, and let 
 $[\tilde{t_i}]$ be its Teichm\"uller lift. 
Let 
$$u_i =t_i \otimes 1 -1\otimes [\tilde{t_i}] \in \O_K\otimes_\ZZ W(\O_C^\flat).$$
 Let 
 $$\theta: W(\O_C^\flat) \to \oc$$
  be the usual Fontaine's map.
\end{notation} 

 \begin{notation}
 By scalar extension,  the $\theta$-map induces a map 
$$\theta_K:  \OK\otimes_{\mathbb{Z}} W(\O_C^\flat) \to \O_{C}.$$ Let 
\begin{equation*}
\Ainf(\O_C/\O_K): = \projlim_{n \geq 0}  \left( \OK\otimes_{\mathbb{Z}} W(\O_C^\flat) \right)/\left(  \theta_K^{-1}(p\O_C)  \right)^n. 
\end{equation*}
Then the $\theta_K$-map extends to 
$$\theta_K: \Ainf(\O_C/\O_K)  \to \O_{C} \quad \text{ and } \quad \theta_K: \Ainf(\O_C/\O_K) \otimes_{\Zp} \Qp \to C.$$
 Let 
$$\bdr^+:= \projlim_{n \geq 1} (\Ainf(\O_C/\O_K) \otimes_{\Zp} \Qp)/\left(\Ker \theta_K \right)^n.  $$
There are natural maps $W(\O_C^\flat) \to \Ainf(\O_C/\O_K)$ and $\Ker \theta \to \Ker \theta_K$, and they induce a natural map
$$\bdr^{\nabla +} \to \bdr^+;$$
By \cite[Prop. 2.9]{Bri06}, it further induces an isomorphism of rings
\begin{equation}
\bdr^{\nabla +}[[u_1, \cdots, u_d]] \simeq \bdr^+.
\end{equation}
This implies that $\bdr^+$ is a $\overline{K}$-algebra, and is a local ring with   maximal ideal 
$$m_{\dR}=\Ker \theta_K=(t, u_1, \cdots, u_d),$$
where $t \in \bdr^{\nabla +}$ is the usual element.
 Define a filtration on $\bdr^+$ where 
 $$\mathrm{fil}^i \bdr^+=m_{\dR}^i, \forall i \geq 0.$$
Using $\mathrm{fil}^i$, define another filtration on $\bdr:=\bdr^+[1/t]$ by:
\begin{eqnarray*}
\Fil^0 \bdr &=& \sum_{n=0}^\infty t^{-n}\mathrm{fil}^n \bdr^+ =\bdr^+[\frac{u_1}{t}, \cdots, \frac{u_d}{t}],\\
\Fil^i \bdr &=& t^i \Fil^0 \bdr, \forall i \in \mathbb{Z}.
\end{eqnarray*} 
$\bdr$ carries a natural $G_K$-action such that $\Fil^i \bdr$ is stable under the action.
\end{notation} 

\begin{notation} 
Let $\bht$ be the graded algebra associated to $\Fil^i$ of $\bdr$. $\bht$ is a $C$-algebra and it carries an induced $G_K$-action. Clearly
$$ \bht \simeq \bht^\nabla[\frac{u_1}{t}, \cdots, \frac{u_d}{t}]  \simeq C[t, t^{-1}, \frac{u_1}{t}, \cdots, \frac{u_d}{t} ].$$ 
\end{notation} 

 \begin{notation} 
Let 
$$\theta_{K_{0}}:\oko\otimes_{\mathbb{Z}} W(\O_C^\flat)\rightarrow \mathcal{O}_{ {C} }$$ 
denote the natural extension of $\theta:W(\O_C^\flat) \rightarrow \mathcal{O}_{ {C} }$.
Let  $\acris$  be the $p$-adic completion of the PD-envelope of $\oko\otimes_{\mathbb{Z}} W(\O_C^\flat)$ with respect to $\Ker(\theta_{K_{0}})$.
Define $$\bcris^{+}:=\acris[1/p], \quad \bcris :=\bcris^{+}[1/t].$$
The ring $\bcris$ is  the $K_{0}$-algebra  endowed with natural commuting  $G_{K}$-action and  $\varphi$-action.
Furthermore, by \cite[Prop. 2.47]{Bri06}, there is a natural injection
$$K\otimes_{K_{0}}\bcris\hookrightarrow \bdr, $$ 
and hence $K\otimes_{K_{0}}\bcris$ is endowed with a filtration induced by that of $\bdr$.
By \cite[Prop. 2.39]{Bri06}, there is an isomorphism of rings
\begin{equation}\label{239}
f:\text{$p$-adic completion of } \ \acris^\nabla \langle u_{1},\ldots,u_{d}\rangle_{\mathrm{PD}}\simeq \acris
\end{equation} 
where $\langle *\rangle_{\mathrm{PD}}$ denotes PD-polynomials.
 
\end{notation} 

  \begin{notation} (cf. \cite[\S 2]{Moritacrys})
   Fix an element $\wt p=(s^{(n)})\in \O_C^\flat$ such that $s^{(0)}=p$.
Then, the series $\log(\wt{p}/p)$ converges  to an element $\gu$ in $\bdr^{\nabla +}$, and $\gu$ is   transcendental over $\bcris$. Let
$$ \mathbf{A}_\st:=\acris[\gu], \quad  \bst^+:=\bcris^+[\gu], \quad \bst:=\bcris[\gu];$$
all these are regarded as subrings of $\bdr$, and are independent of $\wt p$.
$\bst$ is $G_K$-stable inside $\bdr$, and   $K\otimes_{K_{0}}\bst$ is  endowed with a filtration induced by that of $\bdr$.  
We extend the Frobenius $\varphi$ on $\bcris$ to $\bst$ by setting $\varphi(\gu)=p\gu$. 
Furthermore, define the $\bcris$-derivation $N:\bst\rightarrow \bst$ by $N(\gu)=-1$.
It is easy to verify $N\varphi=p\varphi N$.
Induced by \eqref{239}, we have an isomorphism 
\begin{equation}
f:(\text{$p$-adic completion of } \ \acris^\nabla \langle u_{1},\ldots,u_{d}\rangle_{\mathrm{PD}})[1/p,\gu,1/t]\simeq \bst.
\end{equation}
\end{notation}

\begin{prop}
There are canonical isomorphisms 
$$(\bdr)^{G_K}=K, \quad (\bht)^{G_K}=K, \quad (\bcris)^{G_K}=K_{0}, \quad (\bst)^{G_K}=K_{0}.
$$
\end{prop}
\begin{proof}
See \cite[Prop. 2.16, Prop. 2.50]{Bri06} and \cite[Prop. 2.2]{Moritacrys}.
\end{proof}



\subsection{Fontaine functors} 
Let $V$ be a $p$-adic representation of $G_{K}$.   Define
 $$D_{\ast}(V)=(\mathbf{B}_{\ast}\otimes_{\Qp}V)^{G_K}, \text{ for }  \ast \in \{\dR, \HT, \cris, \st  \}.$$ 
 By \cite{Bri06} for $\ast \in \{\dR, \HT, \cris\}$ and by \cite{Moritacrys} for $\ast=\st$, $D_{\dR}(V)$ (resp. $D_{\HT}(V)$, resp. $D_{\cris}(V)$, resp. $D_{\st}(V)$) is a $K$- (resp.  $K$-, resp.  $K_0$-, resp.  $K_0$-) vector space of dimension $\leq \dim_\Qp V$.

\begin{defn}  
\begin{enumerate}
\item   $V$ is a called a de Rham (resp. Hodge-Tate, resp. crystalline, resp. semi-stable) representation of $G_K$ if the aforementioned $K$- or $K_0$-vector space is of dimension $\dim_\Qp V$.

\item Say $V$ is potentially  de Rham (resp. potentially Hodge-Tate, resp. potentially crystalline, resp. potentially semi-stable) if there exists a finite field extension $L/K$  such that $V|_{G_L}$ is a de Rham (resp. Hodge-Tate, resp. crystalline, resp. semi-stable) representation of $G_{L}$. 
\end{enumerate}
 \end{defn}
 
For the reader's convenience, we summarize the \emph{relations} of these notions in Thm. \ref{thmmorita} and Thm. \ref{Moritacor}. These results will \emph{not} be used in this paper.
 
\begin{thm} \label{thmmorita} \cite{MoritadR, Moritacrys}  Let $V$ be a $p$-adic representation of $G_{K}$. 
\begin{enumerate}

\item $V$ is  de Rham $\Longleftrightarrow$ $V$ is  potentially de Rham  $\Longleftrightarrow$ $V|_{G_{\bbk}}$ is   de Rham  (as a representation of     $G_{\bbk}$, similar in the following).

\item $V$ is  Hodge-Tate $\Longleftrightarrow$ $V$ is  potentially Hodge-Tate  $\Longleftrightarrow$ $V|_{G_{\bbk}}$ is  Hodge-Tate.

\item $V$ is  potentially crystalline  $\Longleftrightarrow$ $V|_{G_{\bbk}}$ is  potentially crystalline.

\item \label{item4morita} $V$ is  potentially semi-stable  $\Longleftrightarrow$ $V|_{G_{\bbk}}$ is  potentially semi-stable.
\end{enumerate}
\end{thm} 
\begin{proof}
Items (1) and (2) are proved in \cite{MoritadR}; Items (3) and (4) are proved in \cite{Moritacrys}.
\end{proof}



\begin{thm}\label{Moritacor}
Let $V$ be a $p$-adic representation of $G_{K}$. Then $V$ is  de Rham  $\Longleftrightarrow$ $V$ is  potentially semi-stable. (Hence conditions in Items (1) and (4) of Thm. \ref{thmmorita} are all equivalent.)
\end{thm} 
\begin{proof}
The first proof is due to Morita: it is  an easy consequence of Items (1) and (4) of Thm. \ref{thmmorita}, together with Berger's $p$-adic local  monodromy theorem  \cite{Ber02} (for de Rham representations of $G_{\bbk}$).
Alternatively, another proof which works even when $[k:k^p]=+\infty$ is supplied by \cite{Ohk13}.
\end{proof} 

 \subsection{Weakly admissible implies admissible}
 In this subsection, we prove the main result of this section, namely the Colmez-Fontaine theorem in the imperfect residue field case. In order to do so, we need to introduce connection operators on various rings and modules.
 
\begin{notation}
  Let
$$ \wh{\Omega}_\oko:= \projlim_{n >0} \Omega^1_{\mathcal{O}_{K_0}/\mathbb Z}/p^n\Omega^1_{\mathcal{O}_{K_0}/\mathbb Z}$$
be the $p$-adically continuous K\"ahler differentials, and let 
$$\wh\Omega_{K_0}=\wh{\Omega}_\oko\otimes_\oko K_0.$$ 
$\wh{\Omega}_\oko$ is a free $\oko$-module with a set of basis $\{d\log(t_i)\}_{1\leq i \leq d}$, where $t_1, \cdots, t_d \in \oko$ are introduced in \S \ref{sec21} as a lift of a $p$-basis of $k_K$.
Let 
$$d: \oko \to \wh{\Omega}_{\oko}$$ 
be the canonical differential map.
For $D$ a $\oko$-module, a connection on $D$ is an additive map 
$$\nabla: D \to D\otimes_{\oko} \wh{\Omega}_\oko$$ satisfying Leibniz law with respect to $d$.
Let $\wh{\Omega}_\ok$ and $d: \ok \to \wh{\Omega}_\ok$ be similar defined; then  one can also define connections on $\ok$-modules.
\end{notation}

\begin{notation} \label{drnabla}
Recall that we have $\bdr^+ =\bdr^{\nabla +}[[u_1, \cdots, u_d]]$. For $1\leq i \leq d$, let $N_i$ be the unique $\bdr^{\nabla +}$-derivation of $\bdr^+$ such that 
$$N_i(u_j)=\delta_{i, j}u_j$$ 
where $\delta_{i, j}$ is the Kronecker symbol; hence in particular we have $N_i(t)=0, \forall i$, and thus $N_i$ extends to a $\bdr^{\nabla}$-derivation of $\bdr$. Define a map
\begin{equation}
\nabla:\bdr \to \bdr\otimes_\ko \wh{\Omega}_\ko, \quad x \mapsto \sum_{i=1}^d N_i(x)\otimes d \log(t_i);
\end{equation}
the connection is integrable as $N_i$'s commute with each other. Note that 
$$(\bdr)^{\nabla=0} =\bdr^\nabla,$$
which explains the notation  ``$\bdr^\nabla$" in \S \ref{no222}.
We list some basic properties of $\nabla$.
\begin{enumerate}
\item By \cite[Prop. 2.23]{Bri06}, $\nabla$ satisfies Griffith transversality, i.e., 
$$\nabla(\Fil^r \bdr) \subset \Fil^{r-1}\bdr \otimes_\ko \wh{\Omega}_\ko.$$
\item By \cite[Prop. 2.24]{Bri06}, $\nabla$ commutes with $G_K$-action.
\item By \cite[Prop. 2.25]{Bri06}, $\nabla|_K$ is precisely the canonical differential $d: K \to K \otimes_\ko \wh{\Omega}_\ko.$
\end{enumerate}
\end{notation}

\begin{notation}\label{n243}
One can easily check (cf. \cite[p. 950]{Bri06}) that via $\bcris \into \bdr$,  $\nabla$ on  $\bdr$ induces an integrable connection
$$\nabla: \bcris \to \bcris \otimes_\ko \wh{\Omega}_\ko.$$
Since $\gu \in \bdr^{\nabla +}$ and hence $\nabla(\gu)=0$, we  further have an integrable connection
$$\nabla: \bst \to \bst \otimes_\ko \wh{\Omega}_\ko.$$
We have
$$ (\bcris)^{\nabla=0} =\bcris^\nabla, \quad  (\bst)^{\nabla=0} =\bst^\nabla. $$ 
By \cite[Prop. 2.58]{Bri06}, $\varphi \nabla= \nabla\varphi$ over $\bcris$ (and hence also over $\bst$); here the $\varphi$-action on $\wh{\Omega}_\ko$ extends the $\varphi$-action on $\ko$ and $\varphi(dt_i)=d(\varphi(t_i))$.
In addition, we obviously have $N \nabla= \nabla N$ over $\bst$, where $N=0$  on $\wh{\Omega}_\ko$.
\end{notation}

\begin{defn} \cite[Def. 4.3]{Bri06}
\begin{enumerate}
\item Given a $p$-adic separated and complete $\oko$-module $\mathcal{D}$, a connection  $\nabla: \mathcal{D} \to \mathcal{D}\otimes_{\oko} \wh\Omega_\oko$ is called \emph{quasi-nilpotent} if the induced connection
$$ \nabla: \mathcal{D}/p\mathcal{D} \to \mathcal{D}/p\mathcal{D} \otimes_\oko \wh{\Omega}_\oko$$
is nilpotent.

\item  Given a $\ko$-vector space $D$, a connection  $\nabla: D \to D\otimes_{K_0} \wh\Omega_\ko$ is called \emph{quasi-nilpotent} if there exists an $\oko$-submodule $\mathcal{D}\subset D$ which is $p$-adic separated and complete such that $\mathcal{D}[1/p]=D$ and 
$$\nabla(\mathcal{D}) \subset \mathcal{D}\otimes_\oko \wh{\Omega}_\oko,$$
and such that the restricted connection $\nabla \mid_{\mathcal D}$ is quasi-nilpotent as in Item (1) above.
\end{enumerate}
\end{defn}

\begin{defn}\label{deffilnmod}
Let $\MFnabla$ be the category of filtered $(\varphi,N)$-modules with connections which consists of  finite dimensional $K_0$-vector spaces $D$ equipped with
\begin{enumerate}
\item a $\varphi_{K_0}$-semi-linear Frobenius $\varphi: D \to D$ such that \emph{$\varphi(D)$ generates $D$ over $K_0$};
\item a monodromy $N: D \to D$, which is a $K_{0}$-linear map such that $N\varphi=p\varphi N$;
\item   an integrable quasi-nilpotent connection $\nabla: D \to D\otimes_{\ko} \wh{\Omega}_\ko$, which commutes with $\varphi$ and $N$ (here on the right hand side, $\varphi=\varphi_D\otimes \varphi_{\wh{\Omega}_\ko}$  and $N=N_D\otimes 1$, cf. \S \ref{n243} for $\varphi$ and $N$ over $\wh{\Omega}_\ko$); 
\item a decreasing, separated and exhaustive filtration $(\Fil^{i}D_{K})_{i\in\mathbb{Z}}$ on $D_{K}=D\otimes_{K_0} K$, by  $K$-vector subspaces, such that 
$$\nabla(\Fil^{i}D_{K}) \subset \Fil^{i-1}D_{K}, \forall i,$$
where $\nabla: D_K \to D_K\otimes_K \wh{\Omega}_K$ is the induced connection (satisfying Leibniz law with respect to $d: K \to \wh{\Omega}_{K}$).
\end{enumerate}
\end{defn}

\begin{rem}\label{remdimp}
When $K$ has imperfect residue field, $\varphi$ on $K_0$ is not necessarily bijective; hence we need to require ``$\varphi(D)$ generates $D$ over $K_0$" in Def. \ref{deffilnmod}. This implies that
$$ K_0\otimes_{\varphi, K_0}D \xrightarrow{1\otimes \varphi} D$$
is a  $K_0$-linear  isomorphism.
\end{rem}

\begin{prop}
Let $\Rep_{\Qp}^{\st}(G_K)$ be the category of semi-stable representations of $G_K$.
 Then $D_\st$ induces a functor
$$  \Rep_{\Qp}^{\st}(G_K) \to \MFnabla.$$
\end{prop}
\begin{proof}
The $N=0$ (i.e., crystalline) case is proved in  \cite[Prop. 4.19]{Bri06}. The semi-stable case follows from similar argument; in particular, to show that the induced connection on $D_\st(V)$ is quasi-nilpotent, one simply replaces the use of $\acris$ in \emph{loc. cit.} by $\mathbf{A}_\st$.
\end{proof}

\begin{defn}
\begin{enumerate}
\item  Let $\MFa \subset \MFnabla$ be the   sub-category defined by the essential image  of $D_\st$; the objects in $\MFa$ is called \emph{admissible} objects.
\item For $D\in \MFnabla$, one can define the Newton number $t_N(D)$ and the Hodge number $t_H(D)$ with the usual recipe, cf. \cite[\S 4.1]{Bri06}. $D$ is called \emph{weakly admissible} if $t_N(D)=t_H(D)$ and $t_N(D')\leq t_H(D')$ for any sub-object (in the category $\MFnabla$) $D' \subset D$. Denote the sub-category of weakly admissible objects as $\MFnablawa$.
\end{enumerate}
 \end{defn}

\begin{theorem}\label{thmCF}
The functor $D_\st$ induces an equivalence of categories 
$$\Rep_{\Qp}^{\st}(G_K) \simeq \MFnablawa.$$ 
A quasi-inverse is given by
$$V_{\st}(D)   =   (\bst \otimes_{K_0} D)^{\varphi=1, N=0, \nabla=0} \cap \Fil^0(D_K \otimes_K \bdr).  $$  
\end{theorem}
\begin{proof}
The proof follows   the same strategy of   \cite[Thm. 4.34]{Bri06} (the crystalline case).
Here, we content ourselves by giving a sketch of the argument; in particular, we point out  how the many \emph{ingredients} in \cite[\S 4.2]{Bri06} still hold in the semi-stable case. In the following, let
$$D \in \MFnablawa, \quad   V=V_{\st}(D). $$ 

\textbf{Step 1:}
 By the semi-stable version (see later) of \cite[Prop. 4.28]{Bri06}, we know that $\dim_{\Qp} V <\infty$, and $V$ is semi-stable. Furthermore, if $D'=D_\st(V)$, then $D' \subset D$ is a sub-object in $\MFnabla$.
Now, let us sketch why the semi-stable version  of \cite[Prop. 4.28]{Bri06} holds:
\begin{enumerate}
\item  The semi-stable version of \cite[Thm. 4.24]{Bri06} obviously holds; namely, $D_\st$ induces an equivalence of Tannakian categories
$$\Rep^\st_\Qp(G_K) \xrightarrow{\simeq}  \MFa,$$
and $V_\st$ is a quasi-inverse.

\item The ``semi-stable version" of \cite[Prop. 4.26]{Bri06} holds: in fact, an object $E \in \MFnabla$ of dimension 1  automatically satisfies $N=0$! Hence we automatically have $V_\st(E)=V_\cris(E)$.

\item The semi-stable version of \cite[Prop. 4.27]{Bri06} holds: namely, if $E \in \MFnabla$ is admissible, then it is weakly admissible.   The proof relies heavily on  \cite[Prop. 4.26]{Bri06}, which is still applicable as we just mentioned.
\end{enumerate}
 
\textbf{Step 2:} Let $D_\bbko :=\bbk_0 \otimes_{K_0}D$, then we naturally get an object in $\MFkpf$,
where $\MFkpf$ is the category defined  ``using $\bbk$ instead of $K$" in Def. \ref{deffilnmod}; note that $\nabla$-operators disappear  because $\wh{\Omega}_{\bbk_0}=0$. 
Note that \emph{a priori}, it is not clear if $D_\bbk$ is weakly admissible in $\MFkpf$ (although we will see later it is). We claim that there is an isomorphism of vector spaces:
\begin{equation}\label{eqbstkpfo}
V \xrightarrow{\simeq} (\bbst \otimes_{\bbko} D_\bbko)^{\varphi=1, N=0} \cap \Fil^0 (\bbdr \otimes_\bbk D_\bbk) = \mathbb{V}_{\st}(D_\bbko),
\end{equation}
where $\mathbb{V}_{\st}$ is the obvious functor $\MFkpf \to \Rep_\Qp(G_\bbk)$.
This is the semi-stable version of one of the displayed equations in \cite[Thm. 4.34]{Bri06}, and we sketch the proof here:
\begin{enumerate}
\item By \cite[Cor. 4.31]{Bri06}, we have an isomorphism
\begin{equation} \label{eqder}
\Fil^0(\bdr\otimes_K D_K)^{\nabla=0} \simeq \Fil^0(\bbdr \otimes_{K} D_K)=\Fil^0 (\bbdr \otimes_\bbk D_\bbk).
\end{equation}  
Note that \cite[Cor. 4.31]{Bri06} only concerns the de Rham period rings and $\nabla$-structure on modules; it has nothing to do with $\varphi$- or $N$-structures.
\item The semi-stable version of \cite[Prop. 4.32]{Bri06} and hence \cite[Cor. 4.33]{Bri06} still holds. To prove the semi-stable version of \cite[Prop. 4.32]{Bri06}, it suffices to show that for any $\nabla$-modules over $K_0$, we have
\begin{equation}\label{eq2}
\Hom_{K_0, \nabla}(D, \bst) \simeq \Hom_{K_0}(D, \bbst).  
\end{equation}
Note that we have decompositions
$$\bst=\oplus_{i \geq 0} (\bcris \cdot \mathfrak{u}^i), \quad \bbst=\oplus_{i \geq 0} (\bbcris \cdot \mathfrak{u}^i),  $$ 
which are both $K_0$-linear; the first decomposition is $\nabla$-stable since $\nabla(\mathfrak u)=0$. Hence it suffices to show that the decomposed (``crystalline") pieces of \eqref{eq2} are isomorphic to each other, namely, 
$$\Hom_{K_0, \nabla}(D, \bcris \cdot \mathfrak{u}^i) \simeq \Hom_{K_0}(D, \bbcris \cdot \mathfrak{u}^i), \quad \forall i;  $$
these follow from \cite[Prop. 4.32]{Bri06}.
Note that \eqref{eq2}   implies the full semi-stable version of \cite[Prop. 4.32]{Bri06}. Thus the  semi-stable version of \cite[Cor. 4.33]{Bri06} holds by some obvious duality argument; namely we have
\begin{equation}\label{eq33}
(\bst \otimes_{K_0} D)^{\varphi=1, N=0, \nabla=0} \simeq (\bbst \otimes_{\bbko} D_\bbko)^{\varphi=1, N=0}.
\end{equation}

\item Finally, \eqref{eqbstkpfo} holds by using \eqref{eqder} and \eqref{eq33}.
\end{enumerate}

\textbf{Step 3:} With Step 1 and Step 2 established, one  can then use exactly the same argument as in the final two paragraphs of \cite[Thm. 4.34]{Bri06} (which uses \cite[Thm. 4.3.(ii)]{CF00}) to conclude that $D$ is admissible.
\end{proof}



\section{Modification of Fontaine modules}\label{sec3}
In this section, we follow the idea of Kisin -- modified by Brinon-Trihan in the imperfect residue field case --  to study modification of Fontaine modules. The main theorem says that we can construct a fully faithful functor from the category of weakly admissible filtered $(\varphi, N)$-modules with connections to a certain category of Breuil-Kisin modules equipped with $(\varphi, N, \nabla)$-operators. The constructed Breuil-Kisin modules will be used in our integral theory in \S \ref{sec4}.
   
\subsection{Fontaine modules and Kisin's $\mathcal A$-modules}   
\begin{notation}   
Let
$$
\mathcal{A} = \{f(u)= \sum_{i=0}^{+\infty} a_i u^i, a_i \in K_0 \mid   f(u) \text{ converges, } \forall u \in \mathfrak{m}_{\mathcal{O}_{\overline{K}}}   \},
$$
where $\mathfrak{m}_{\mathcal{O}_{\overline{K}}}  $ is the maximal ideal of $\mathcal{O}_{\overline{K}}$;
i.e., it consists of series that converge on the entire open unit disk defined over $K_0$.
(The ring is denoted as $\mathcal{O}=\mathcal{O}^{[0, 1)}$ in \cite{Kis06}; we follow the notation of \cite{BT08} and use $\mathcal{A}$ here, because we use  ``$\ok , \oko$" etc to denote rings of integers in this paper.)
Let $\varphi$ be the Frobenius operator extending $\varphi$ on $K_0$ such that $\varphi(u)=u^p$.
Let $\pi \in K$ be a \emph{fixed} uniformizer, and let $E(u)$ be its minimal polynomial over $K_0$. Define an element
$$\lambda =\prod_{n=0}^\infty \varphi^n(\frac{E(u)}{E(0)}) \in \mathcal{A}.$$
Let $N_\nabla$ be the $K_0$-linear differential operator $u\lambda\frac{d}{du}$ on $\mathcal{A}$.
\end{notation}  

\begin{defn} \label{defaom}
 Let $\Mod_{ \mathcal{A}}({\varphi, N_\nabla, \nabla})$ be the category  of the following data: 
\begin{enumerate}
\item $M$ is a  finite free $\mathcal{A}$-module;
\item $\varphi : M \to M$ is a $\varphi_{\mathcal A}$-semi-linear morphism  such that   the cokernel of $1 \otimes \varphi : \varphi ^*M \to M $ is killed by $E(u)^h$ for some $h \in \mathbb{Z}^{\geq 0}$;
\item $N_\nabla: M \to M$ is a  map such that $N_\nabla(fm)=N_\nabla(f)m+fN_\nabla(m)$ for all $f\in \mathcal A$ and $m \in M$, and $N_\nabla\varphi=\frac{pE(u)}{E(0)} \varphi N_\nabla$;
\item $\nabla: M/uM \to M/uM \otimes_{{K_0}} \wh{\Omega}_\ko$ is an integrable quasi-nilpotent connection which commutes with $\varphi$ and $N_\nabla$ (where $N_\nabla=0$ on $\wh{\Omega}_\ko$).
\end{enumerate}
\end{defn}
 
\begin{notation}  \label{n313}
\emph{Fix} a system of elements $\pi_n \in \overline{K}$ such that $\pi_0=\pi$ and $\pi_{n+1}^p=\pi_n, \forall n \geq 0$.
For $n \geq 0$, let $K_{n+1}=K(\pi_{n})$ (hence $K_1=K$). Let
$$\gs =\oko[[u]] \subset \mathcal{A}.$$
Let $\wh{\gs}_n$ be the completion of $K_{n+1}\otimes_{W(k)} \gs$ at the maximal ideal $(u-\pi_n)$; $\wh{\gs}_n$  is equipped with its $(u-\pi_n)$-adic filtration, which extends to a filtration on the quotient field  $\wh{\gs}_n[1/(u-\pi_n)]$.
There is a natural $K_0$-linear map 
$$\mathcal A \to \wh{\gs}_n$$ 
simply by sending $u$ to $u$. 
Let $\ell_u$ be a formal variable, which behaves like $\log u$.
 We can extend the map  $\mathcal A \to \wh{\gs}_n$ to  $\mathcal{A}[\ell_u] \to \wh{\gs}_n$ which sends $\ell_u$ to
$$ \sum_{i=1}^\infty (-1)^{i-1}i^{-1}(\frac{u-\pi_n}{\pi_n})^i \in   \wh{\gs}_n.$$
We can  naturally extend $\varphi$ to  $\mathcal{A}[\ell_u]$ by setting $\varphi(\ell_u)=p\ell_u$, and extend $N_\nabla$ to  $\mathcal{A}[\ell_u]$ by setting $N_\nabla(\ell_u)=- \lambda$.
 Finally, let $N$ be the $\mathcal{A}$-derivation on $\mathcal{A}[\ell_u]$ such that $N(\ell_u)=1$.
 \end{notation}
 

\begin{notation}
The above construction works for any CDVF $K$, hence in particular for $\bbk$. Hence for example, we can define
$$
\mathcal{A}_\bbk = \{f(u)= \sum_{i=0}^{+\infty} a_i u^i, a_i \in \bbk_0 \mid   f(u) \text{ converges }, \forall u \in \mathfrak{m}_{\O_{\overline{\bbk}}}   \}.
$$
Furthermore, since $\pi \in K$ is also a uniformizer of $\bbk$, hence we can keep using the elements $\pi_n$ (regarded as elements in $\overline{\bbk}$) and the polynomial $E(u)$ (regarded as the minimal polynomial of $\pi$ over $\bbk_0$). Thus when we work with $\bbk$, we get exactly the same $\lambda$, and the same $\varphi, N_\nabla$ operators.
We then can define $\Mod_{\mathcal{A}_\bbk}({\varphi, N_\nabla})$ (without $\nabla$-operators as $\wh{\Omega}_\bbk=0$). Then we can define  $\gs_\bbk,  \bbk_n,  \wh{\gs_\bbk}_n$ etc. together with the various operators. The formal variable $\ell_u$ behaves in the same way in this case.
\end{notation}

\begin{construction}
Let $$\MFkpfgeq \subset \MFkpf$$ be the sub-category  consisting of objects with $\Fil^0 \bbd_\bbk=\bbd_\bbk$; here $\MFkpf$ is the category of usual Fontaine modules for $\bbk$ (which has perfect residue field).
For $\mathbb D \in  \MFkpfgeq$, write $\iota_n$ for the following composite map:
\begin{equation}
\mathcal{A}_\bbk[\ell_u] \otimes_{\bbK_0} \bbd \xrightarrow{\varphi_{\bbk_0}^{-n}\otimes \varphi^{-n}} \mathcal{A}_\bbk[\ell_u]\otimes_{\bbK_0} \bbd \rightarrow 
\wh{\gs}_{\bbk,n}\otimes_{\bbK_0} \bbd 
= 
\wh{\gs}_{\bbk,n}\otimes_{\bbK} \bbd_K
\end{equation}
Here
\begin{itemize}
\item $\varphi_{\bbk_0} : \mathcal{A}_\bbk[\ell_u] \to \mathcal{A}_\bbk[\ell_u] $ is the map which acts on $\bbk_0$ by $\varphi$ and fixes $u$ and $\ell_u$;
\item note that the map  $\varphi^{-n}: \mathbb D \to \mathbb D$ is well-defined because $\bbk$ has \emph{perfect} residue field (compare with Rem. \ref{remdimp});
\item  the second map is induced from the map  $\mathcal{A}_\bbk[\ell_u] \to \wh{\gs}_{\bbk, n}$. 
\end{itemize} 
The composite map extends to
\begin{equation}
\iota_n: \mathcal{A}_\bbk[\ell_u, 1/\lambda] \otimes_{\bbK_0} \mathbb D \rightarrow \wh{\gs}_{\bbk, n}[1/(u-\pi_n)]\otimes_{\bbK} \mathbb D_\bbK
\end{equation}
Now, set
\begin{equation}\label{eqdefmd}
\mathbb M(\mathbb D): = \{ x\in (\mathcal{A}_\bbk[\ell_u, 1/\lambda] \otimes_{\bbK_0} \mathbb D )^{N=0}: \iota_n(x)\in \Fil^0  \left( \wh{\gs}_{\bbk,n}[1/(u-\pi_n)]\otimes_{\bbK} \mathbb D_\bbK\right), \forall n\geq 1  \},
\end{equation}
where the $\Fil^0$ in \eqref{eqdefmd} comes from tensor product of two filtrations.
\end{construction}

\begin{theorem} 
\cite[Thm. 1.2.15]{Kis06}
$\mathbb M(\mathbb D)$ is finite free over $\mathcal{A}_\bbk$, is stable under   $\varphi\otimes \varphi$-action and $N_\nabla\otimes 1$-action induced from $\mathcal{A}_\bbk[\ell_u, 1/\lambda] \otimes_{\bbK_0} \mathbb D $. Indeed, $\bbd \mapsto \mathbb M(\mathbb D)$ induces an equivalence of categories
$$ \MFkpfgeq \xrightarrow{\simeq}  \Mod_{ \mathcal{A}_\bbk}({\varphi, N_\nabla}).$$
\end{theorem}
 
 \begin{construction}
 The embedding $K_0 \hookrightarrow \bbk_0$ extends to $\mathcal{A} \hookrightarrow \mathcal{A}_\bbk$ where $u \mapsto u$; this embedding is equivariant with the operators $\varphi,   N_\nabla$ defined above. The embedding also extends to $\mathcal{A}[\ell_u] \hookrightarrow \mathcal{A}_\bbk[\ell_u]$ as well as to $\wh{\gs}_n \hookrightarrow \wh{\gs_\bbk}_n$, which are compatible with all the various structures discussed above. 
 Hence if $D \in \MFnablageq$, then we can regard $\mathcal{A}[\ell_u, 1/\lambda]\otimes_\ko D$ as a submodule of $\mathcal{A}_\bbk[\ell_u, 1/\lambda] \otimes_{\bbK_0} \mathbb D  $ compatible with various structures.
Define 
$$M(D):= \left(\mathcal{A}[\ell_u, 1/\lambda] \otimes_{K_0} D\right) \cap \mathbb M(D_{\bbk}).$$
 \end{construction}
 
 \begin{thm}\label{thmmod}
$M(D)$ is finite free over $\mathcal{A}$, is stable under  $\varphi$-action and $N_\nabla$-action induced from $\mathbb M(D_{\bbk})$, and there is a $(\varphi, N)$-equivariant isomorphism
 $$M(D)/uM(D) \simeq D,$$ 
 where $N$ acts on $M(D)/uM(D)$ by $N_\nabla \pmod u$. Equipping $\nabla$ on $M(D)/uM(D)$ from that on $D$, this makes $M(D)$ an object in $\Mod_{ \mathcal{A}}({\varphi, N_\nabla, \nabla})$. Indeed, this construction  induces an equivalence of categories
$$ \MFnablageq \xrightarrow{\simeq} \ModA.$$ 
\end{thm} 
\begin{proof}
This is  analogue of  \cite[Thm. 1.2.15]{Kis06}, and the $N=0$ case is proved \cite[Prop. 4.11]{BT08}. The general case here follows the same argument as in \emph{loc. cit.} (by keeping track of the $N$- and $N_\nabla$-operators).
\end{proof}

Let
\begin{equation}\label{eqbrig}
\begin{split}
\mathcal R: =\{f(u)= \sum_{i=-\infty}^{+\infty} a_i u^i, a_i \in K_0,   f(u) \text{ converges } \\
 \text{ for all } u \in \overline{K} \text{ with } 0<v_p(u)<\rho(f) \text{ for some } \rho(f)>0\}.
\end{split}
\end{equation} 
be the Robba ring (with coefficients in $K_0$). 
 
 \begin{defn}
Let $\ModAO$ be the subcategory of $\ModA$ consisting of objects $M$ such that $\mathcal{R}\otimes_{\mathcal A}M$ is \emph{pure of slope $0$} in the sense of Kedlaya (cf. \cite{Ked04, Ked05}.
\end{defn}

\begin{thm}\label{thmfful}
The functor in Thm. \ref{thmmod} induces a fully faithful functor
$$ \MFnablawageq  \to \ModAO. $$
\end{thm}
\begin{proof}
This is the ``analogue" of \cite[Thm. 1.3.8]{Kis06}, and the $N=0$ case is proved \cite[Prop. 4.14]{BT08}. The general case here follows the same argument as in \emph{loc. cit.}; in particular, let us mention that the proof makes uses of the ``weakly admissible implies admissible" theorem in Thm. \ref{thmCF}.
Note that in the perfect residue field case \cite[Thm. 1.3.8]{Kis06}, the functor is indeed an \emph{equivalence of categories}, which we do not expect in the general case. 
\end{proof}

\begin{defn}
Let $\Mod_{\mathcal A}(\varphi, N, \nabla)$ be the category of the following data:
\begin{enumerate}
\item $M$ is a  finite free $\mathcal{A}$-module;
\item $\varphi : M \to M$ is a $\varphi_{\mathcal A}$-semi-linear morphism  such that   the cokernel of $1 \otimes \varphi : \varphi ^*M \to M $ is killed by $E(u)^h$ for some $h \in \mathbb{Z}^{\geq 0}$;
\item $N: M/uM \to M/uM$ is a $K_0$-linear map such that $N \varphi=p \varphi N$.
\item $\nabla: M/uM \to M/uM \otimes_{{K_0}} \wh{\Omega}_\ko$ is an integrable quasi-nilpotent connection which commutes with $\varphi$ and $N $.
\end{enumerate}
Let $\Mod_{\mathcal A}^{0}(\varphi, N, \nabla)$ be the sub-category consisting of modules which are pure of slope 0 (after tensoring with $\mathcal R$).
\end{defn}

Given $M \in \ModA$, we can construct a module  $\wt M \in \Mod_{\mathcal A}(\varphi, N, \nabla)$ by taking $\wt M =M$ equipped with the same $\varphi$ and $\nabla$,   and take $N$ to be the reduction of $N_\nabla$ modulo $u$.
\begin{prop} \label{3111}
The map $M \mapsto \wt M$ above induces a  fully faithful functor:
$$ \ModA \to  \Mod_{\mathcal A}(\varphi, N, \nabla),$$
which then induces a fully faithful functor 
$$ \ModAO \to  \Mod_{\mathcal A}^{0}(\varphi, N, \nabla).$$
\end{prop}
\begin{proof}
When $K$ has perfect residue field, this is \cite[Lem. 1.3.10(2)]{Kis06}. The general case here follows exactly the same proof: the additional $\nabla$-operators here cause  no trouble, as they remain unchanged under the functor.
\end{proof}

\subsection{Relation with modules over $\gs$ and $S$}
Recall $\gs =\oko[[u]] \subset \mathcal{A},$ it is stable under $\varphi$.

\begin{defn}\label{321}
 Let $\Mod_{\gs}^{\varphi}$ be the category consisting of $(\gm, \varphi)$ where $\gm$ is a finite free $\gs$-module, and $\varphi: \gm \to \gm$ is a $\varphi_\gs$-semi-linear map such that the $\gs$-linear span of $\varphi(\gm)$ contains $E(u)^h\gm$ for some $h\geq 0$. We say that $\gm$ is of $E(u)$-height $\leq h$.
 \end{defn}
 
 \begin{defn}\label{def39}
Let $\Modgsnabla$ be the category of the following data:
\begin{enumerate}
\item $(\gm, \varphi)$ is an object in $\Mod_{\gs}^{\varphi}$;
\item  a $K_0$-linear map $N: \gm/u\gm[1/p] \to \gm/u\gm[1/p]$ such that $N\varphi=p\varphi N$ over $\gm/u\gm[1/p]$;
\item an integral quasi-nilpotent connection $\nabla: \gm/u\gm[1/p] \to \gm/u\gm[1/p]\otimes_{K_0}\wh\Omega_{K_0}$ which commutes with $\varphi$ and $N$.
\end{enumerate}
 Let $\Modgsnabla\otimes_\Zp \Qp$ be   its isogeny category.
 \end{defn}

\begin{theorem}\label{323}
We have the following commutative diagram of categories. Here $\simeq$ indicates an equivalence of categories, and a hooked arrow indicates a fully faithful functor.
\begin{equation}\label{tik1}
 \begin{tikzcd}
\MFnablageq \arrow[dd, "  \simeq"] &  & \MFnablawageq \arrow[dd, " ", hook] \arrow[ll, " "', hook'] \arrow[rr, " (\ast)", hook] &  & \Modgsnabla\otimes_\Zp \Qp \arrow[dd, "\Theta \simeq "] \\
                                     &  &                                                                                         &  &                                                     \\
\ModA                                &  & \ModAO \arrow[ll, " "', hook'] \arrow[rr, " ", hook]                                &  & {\Mod_{\mathcal A}^{0}(\varphi, N, \nabla)}        
\end{tikzcd}
\end{equation}
\end{theorem}

\begin{proof}
With  Thm. \ref{thmmod}, Thm. \ref{thmfful} and Prop. \ref{3111} established, it suffices to discuss the arrow labeled as $\Theta$; the equivalence $\Theta$ would induce the fully faithful functor labelled as $(\ast)$. 

The functor  $\Theta$ is simply defined via
$$\gm \mapsto \gm \otimes_\gs \mathcal A.$$
  The equivalence $\Theta$ is the analogue of \cite[Lem. 1.3.13]{Kis06}, and the $N=0$ case is proved \cite[Prop. 4.17]{BT08}. The general case here follows the same argument as in \emph{loc. cit.}; in particular, as mentioned in \emph{loc. cit.}, the proof makes use of slope filtration theorem of Kedlaya \cite{Ked04, Ked05}, which has no restriction on residue fields.
\end{proof}

\begin{notation}\label{n324}
Let $S$ be the $p$-adic completion of the PD envelope of $\O_{K_0}[u]$ with respect to the ideal $(E(u))$. Explicitly, 
$$ S =\{ x=\sum_{i \geq 0} a_i \frac{E(u)^i}{i!} \mid  a_i \in \O_{K_0}[u], \lim_{i \to \infty} v_p(a_i)=+\infty  \} \subset K_0[[u]].$$ 
\end{notation}

\begin{prop}\label{p325}
Suppose $D \in \MFnablawageq$, and it maps to $\gm$ (up to isogeny) via the functor labelled as $(\ast)$ in Thm. \ref{323}. Then there is an $\varphi$-equivariant isomorphism
\begin{equation} \label{eqdsms}
D\otimes_\ko S[1/p] \simeq \gm\otimes_{\varphi, \gs} S[1/p].
\end{equation}
\end{prop}
\begin{proof}
This follows \cite[Lem. 1.2.6]{Kis06} (cf. also \cite[\S 3.2]{Liu08}) when $K$ has perfect residue field, 
and follows the argument 
in \cite[Prop. 5.7]{BT08} in the general case. Note that in \cite[Prop. 5.7]{BT08}, 
it was assumed that $ \gm\otimes_{\varphi, \gs} S$ is an object in ``$\mathrm{MF}^{\BT}_S(\varphi, \nabla)$" (cf. Def. \ref{d521}) and in particular the $E(u)$-height of $\gm$ is $\leq 1$; but this assumption is irrelevant.
\end{proof}

\section{Integral $p$-adic Hodge theory}\label{sec4}
The main goal in this section is to study \emph{integral} semi-stable representations. In \S \ref{sub42}, we study the the relation between \'etale $\varphi$-modules, Breuil-Kisin modules, and Galois representations; in  \S \ref{new43} we strengthen some of these results in the semi-stable case.
In \S \ref{sub43}, we construct a fully faithful functor from the category of integral  semi-stable representations to the category of {Breuil-Kisin  $ G_K$-modules}. 
Some  of the results in \S \ref{sub42}, as well as in \S \ref{sec5} and \S \ref{secpdiv}  (but \emph{not} in \S \ref{sub43})  were in fact already established in the \emph{relative} case in \cite{Kim15}; hence in \S \ref{sub41}, we first review some notions from \cite{Kim15} which will be mentioned later.

\subsection{Relation with Kim's work}\label{sub41}
We first review some notions from \cite{Kim15}.

\begin{assumption}\label{assr}
Let $R$ be a $p$-adically complete and separated flat $\Zp$-algebra. We list some assumptions on $R$.

\begin{enumerate}
\item[(i)] ($p$-basis, cf. \cite[\S 2.2.1]{Kim15}). Assume that $R \cong R_0[u]/E(u)$, where $R_0$ is a $p$-adic flat $\Zp$-algebra such that $R_0/(p)$ locally admits a finite $p$-basis, and 
\[E(u) = p + \sum_{i=1}^e a_i u^i \] 
for some integer $e>0$, with $a_i\in R_0$ and $a_e\in R_0^\times$ (where $R_0^\times \subset R_0$ are the units).  Let $\varpi\in R$ denote the image of $u\in R_0[u]$.

\item[(i)']  ($p$-basis + Cohen subring, cf. \cite[\S 2.2.1]{Kim15}). In addition to  (i), assume there is a Cohen subring $W\subset R_0$ such that $E(u)\in W[u]$.

\item[(ii)] ($p$-basis+formally finite-type, cf. \cite[\S 2.2.2]{Kim15}). In addition to   (i), assume  that $R$ is $J_R$-adically separated and complete for some finitely generated ideal $J_R$ containing $\varpi$, and $R/J_R$ is finitely generated over some field $k$.

\item[(iii)] ($p$-basis+formally finite-type + refined almost \'etaleness, cf. \cite[\S 2.2.3]{Kim15}). In addition to  (ii), assume $R$ is a domain such that $R[1/p]$ is finite \'etale over $R_0[1/p]$ and we have $\wh\Omega_{R_0} = \bigoplus_{i=1}^d R_0 d T_i$  for some units $T_i\in R_0^\times$. Here, $\wh\Omega_{R_0}$ is the module of $p$-adically continuous K\"ahler differentials. 
\end{enumerate}
\end{assumption}

\begin{rem}
In \cite[\S 2.2.4]{Kim15},  a certain normality assumption is also introduced; it is weaker than Assumption (ii) above. Also, certain local complete intersection assumption is introduced in \cite[\S 2.2.5]{Kim15}; it is irrelevant in the current paper.
\end{rem}

\begin{lemma}\label{lemok}
$\ok$ satisfies \emph{all} the assumptions (including (i)') in \S \ref{assr}.
\end{lemma}
\begin{proof}
For $R=\ok$, we can use the subring $R_0=\oko$, the polynomial $\frac{p}{E(0)}E(u)$ (then $\varpi=\pi$), and $J_R=(\pi)$.
\end{proof}


 \subsection{\'Etale $\varphi$-modules} \label{sub42}
 \begin{notation}\label{nn421}
Recall $\gs=\oko[[u]]$. Let $\OE$ be the $p$-adic completion of $\gs[1/u]$. 
We have defined $\pi_n$ in \S \ref{n313}. Regard $\underline{\pi}=\{\pi_n\}_{n \geq 0}$  as an element in $\O_{C}^\flat$, and let $[\underline{\pi}] \in W(\O_{C}^\flat)$ be its Teichm\"uller lift.
The embedding $i_\varphi: \oko \into W(\mathbf{k})$ induces an $\oko$-linear and $\varphi$-equivariant embedding 
$$\gs  \into W(\O_{C}^\flat)$$ 
by sending $u$ to $[\underline{\pi}]$; it extends to an embedding
$$ \OE \into W({C}^\flat). $$
Let $\O_{\E^\ur} \subset W(C^\flat)$ be the maximal unramified extension of $\O_{\E}$, and let $\O_{\wh{\E}^\ur} $ be its $p$-adic completion. Let $\wh{\gs}^\ur =\O_{\wh{\E}^\ur} \cap W(\O_\mathbb{C}^\flat)$. 
\end{notation}

\begin{defn}\label{421}
An \'etale $\varphi$-module  is a finite free $\OE$-modules $M $  equipped with a $\varphi_{\OE}$-semi-linear endomorphism $\varphi _M : M\to M$ such that $1 \otimes \varphi : \varphi ^*M \to M $ is an isomorphism. Let $\Mod_{\OE}^\varphi$ denote the category of these objects.
\end{defn}

\begin{notation}
Starting from $\bbk$, we can also define $\gs_\bbk=\O_{\bbk_0}[[u]]$. Let $\O_{\E_\bbk}$ be the $p$-adic completion of $\gs_\bbk[1/u]$. 
Similarly as in \S \ref{nn421}, we have  $\O_{\bbk_0}$-linear embeddings 
$$\gs_\bbk \into W(\O_\mathbb{C}^\flat), \quad \O_{\E_\bbk}\into W(\mathbb{C}^\flat),$$
where $\O_\mathbb{C}^\flat$ and $\mathbb{C}^\flat$ are the tiltings, and where we regard  $\underline{\pi}$ as an element in $\O_\mathbb{C}^\flat$.
We can similarly define $\O_{\wh{\E}^\ur_\bbk} $ and $\wh{\gs}_\bbk^\ur$.
There are $\varphi$-equivariant embeddings 
\begin{equation}\label{1}
\begin{tikzcd}
\gs \arrow[r, hook] \arrow[d, hook]      & \wh{\gs}^\ur \arrow[r, hook] \arrow[d, hook]      & W(\O_{C}^\flat) \arrow[d, "{i_\varphi, \simeq}"] \\
\gs_\bbk \arrow[r, hook] \arrow[d, hook] & \wh{\gs}_\bbk^\ur \arrow[d, hook] \arrow[r, hook] & W(\O_\mathbb{C}^\flat) \arrow[d, hook]           \\
\O_{\E_\bbk} \arrow[r, hook]             & \O_{\wh{\E}^\ur_\bbk} \arrow[r, hook]             & W(\mathbb{C}^\flat)                              \\
\OE \arrow[r, hook] \arrow[u, hook]      & \O_{\wh{\E}^\ur} \arrow[r, hook] \arrow[u, hook]  & W({C}^\flat) \arrow[u, "{i_\varphi, \simeq}"']  
\end{tikzcd}
\end{equation} 
Analogous to Def. \ref{321} and Def. \ref{421}, we can define $\Mod_{\gs_\bbk}^\varphi$ and $\Mod_{\O_{\E_\bbk}}^\varphi$. 
\end{notation}

\begin{construction} \label{423c}
In the following, we will explain the content of the following commutative diagram of functors. Here $\simeq$ signifies an equivalence of categories, and a hooked arrow signifies a fully faithful functor.
\begin{equation} \label{12}
\begin{tikzcd}
\Mod_{\gs }^\varphi \arrow[rr, "(1)", hook] \arrow[dd, "(2)", hook'] &  & \Mod_{\gs_\bbk}^\varphi \arrow[rr, "(5)", hook] \arrow[dd, "(4)", hook'] &  & \Rep_{\Zp}(G_{\bbkinfty}) \arrow[dd, "(7)="] \\
                                                                     &  &                                                                          &  &                                           \\
 \Mod_{\O_{\E }}^\varphi \arrow[rr, "(3) \simeq"]                    &  & \Mod_{\O_{\E_\bbk}}^\varphi \arrow[rr, "(6) \simeq"]                     &  & \Rep_{\Zp}(G_{\bbkinfty})                
\end{tikzcd}
\end{equation} 
Here
\begin{itemize}
\item The functors (1)-(4) are defined via scalar extensions in \eqref{1}; they are obviously well-defined and the left square is commutative.
\item The functor (5) is defined via $ \wt\gm\mapsto  (\wt\gm \otimes_{\gs_\bbk} \wh{\gs}_\bbk^\ur)^{\varphi=1}$.
\item The functor  (6) is defined via $ \wt M \mapsto (\wt M \otimes_{\O_{\E_\bbk}} \O_{\wh{\E}^\ur_\bbk})^{\varphi=1} $.
\item Note that $\bbk$ has perfect residue field, hence the commutativity of the right square, as well as the equivalence of (6), the full faithfulness of (4) and (5), are proved in \cite[\S 2.1]{Kis06}.
\end{itemize}
 \emph{It   remains} to show the equivalence of (3) and the full faithfulness of (1), which is carried out in Prop. \ref{prop42} and Prop. \ref{prop43} respectively.
\end{construction}

\begin{prop}\label{prop42}
We have tensor  exact  equivalences  
$$ \Mod_{\OE}^\varphi \xrightarrow{(3)}  \Mod_{\O_{\E_\bbk}}^\varphi \xrightarrow{(6)} \Rep_{\Zp}(G_{\bbkinfty}).$$ 
Given $T \in  \Rep_{\Zp}(G_{\bbkinfty})$,   the corresponding objects $M \in \Mod_{\OE}^\varphi$ and $M_\bbk \in \Mod_{\O_{\E_\bbk}}^\varphi$ are
\begin{equation}\label{eqmmk}
M=(T\otimes_\zp \O_{\wh{\E}^\ur})^{G_{\bbkinfty}}, \quad M_\bbk =  (T\otimes_\zp \O_{\wh{\E_\bbk}^\ur})^{G_{\bbkinfty}}.
\end{equation}
\end{prop}

\begin{rem}
Let $R$ be as in Assumption \ref{assr}. Suppose  
$R$ is a domain, and satisfies Assumptions (i), (i)' and (ii) there. 
Then \cite[Prop. 7.7]{Kim15} proves the \emph{relative} version of (and  hence by Lem. \ref{lemok}, implies) Prop. \ref{prop42}. We nonetheless give a sketch of Kim's proof, to illustrate the useful ideas.
\end{rem}

\begin{proof}[Proof of Prop. \ref{prop42}]
We first introduce some notations, let 
$$\mathbf{E}^+_{\Kinfty}:=\gs/p\gs=k_K[[u]], \quad\mathbf{E}^+_{\bbkinfty}:=\gs_\bbk/p\gs_\bbk=\mathbf{k}[[u]].$$
Note that $\mathbf{E}^+_{\Kinfty}$ and $\mathbf{E}^+_{\bbkinfty}$   have the same perfect closures as $\mathbf{k}$ is the perfect closure of $k_K$; let $\wtE^+_{\bbkinfty}$ be the $u$-adic completion of their common perfect closure. 
Let
$$ \mathbf{E}_{\Kinfty}=\mathbf{E}^+_{\Kinfty}[1/u], \quad \mathbf{E}_{\bbkinfty}=\mathbf{E}^+_{\bbkinfty}[1/u], \quad   \wtE_{\bbkinfty}=\wtE^+_{\bbkinfty}[1/u].$$

By d\'evissage, to prove the proposition, it suffices to prove equivalences of the corresponding $p$-torsion categories. Namely,  it suffices to show equivalences of categories
\begin{equation}\label{eqeq}
\Mod^\varphi_{\mathbf{E}_{\Kinfty}}  \to \Mod^\varphi_{\mathbf{E}_{\bbkinfty}} \to \Rep_{\Fp}(G_{\bbkinfty}),
\end{equation}
where the categories are defined in the obvious fashion.
Note that we have the following isomorphisms of Galois groups
\begin{equation} \label{e424}
G_{\bbkinfty} \simeq G_{\mathbf{E}_{\bbkinfty}} \simeq G_{\wtE_{\bbkinfty}} \simeq G_{\mathbf{E}_{\Kinfty}}.
\end{equation}
Here, the first isomorphism follows from classical theory of field of norms (as $\bbk$ has perfect residue field); the other isomorphisms follow from  \cite[Prop. 5.4.53]{GabberRamero} as both  $\mathbf{E}^+_{\Kinfty}$ and $\mathbf{E}^+_{\bbkinfty}$ are Henselian rings with respect to the ideal generated by $u$. Thus we can apply \cite[Prop. 4.1.1]{Kat72} to conclude equivalences in \eqref{eqeq}.
Finally, from \eqref{e424}, we have
\begin{equation}\label{e425}
(\O_{\wh{\E}^\ur})^{G_{\bbkinfty}}= \O_{ {\E}}, \quad    (\O_{\wh{\E_\bbk}^\ur})^{G_{\bbkinfty}} =\O_{ {\E_\bbk}}.
\end{equation}
This proves \eqref{eqmmk}.
\end{proof}


\begin{prop}\label{prop43}
We have a chain of fully faithful functors:
$$ \Mod_{\gs }^\varphi \xrightarrow{(1)} \Mod_{\gs_\bbk}^\varphi \xrightarrow{(5)} \Rep_{\Zp}(G_{\bbkinfty}).$$
\end{prop}
\begin{proof}
Let $\gm_1, \gm_2 \in  \Mod_{\gs }^\varphi$.
Let 
$$\wt\gm_1, \wt\gm_2 \in \Mod_{\gs_\bbk}^\varphi, \quad 
T_1, T_2\in \Rep_{\Zp}(G_{\bbkinfty}), \quad M_1, M_2 \in \Mod_{\OE}^\varphi, \quad  \wt M_1, \wt M_2 \in \Mod_{\O_{\E_\bbk}}^\varphi$$
  be the corresponding objects. 
  Recall full faithfulness of (5) is known by \cite[Prop. 2.1.12]{Kis06}. The functor (1) is obviously faithful, hence it suffices to show fullness of the composite functor $(5)\circ (1)$.
Given a morphism $T_1 \to T_2$, the equivalences in Prop. \ref{prop42} supplies a unique corresponding morphism $M_1 \to M_2$ and a unique corresponding morphism  $\wt M_1 \to \wt M_2$; the full faithfulness of (5)  supplies a unique corresponding   morphism $\wt\gm_1 \to \wt\gm_2$. Since 
$$\gs =\OE \cap \gs_\bbk \subset \O_{\E_\bbk},$$ 
we have
$$\gm_i =M_i \cap \wt\gm_i \subset \wt M_i.$$ 
 Hence  the desired morphism $\gm_1 \to \gm_2$ is induced by the aforementioned morphisms.
\end{proof}

\begin{notation}
From now on, we denote the fully faithful functor (5) as
$$T_{\gs_\bbk}:  \Mod_{\gs_\bbk}^\varphi \to \Rep_{\Zp}(G_{\bbkinfty}),$$
and denote the  composite fully faithful functor $(5)\circ (1)$ above as
$$T_{\gs}: \Mod_{\gs }^\varphi \to \Rep_{\Zp}(G_{\bbkinfty}). $$
 \end{notation}

\begin{lemma}\label{45}
Let $\gm \in \Mod_\gs^\varphi$,  let $V=T_\gs(\gm)\otimes_\Zp \Qp \in \Rep_\Qp(G_{\bbkinfty})$, and let $M[1/p]=\gm\otimes_\gs \E$ where $\E=\O_\E[1/p]$. 
Then the map $\gn \mapsto T_\gs(\gn)$ induces a bijection between $\varphi$-stable $\gs$-submodules $\gn \in M[1/p]$ such that $\E\otimes_\gs \gn=M[1/p]$ and $\gn/\varphi^\ast(\gn)$ is killed by a power of $E(u)$, and $G_{\bbkinfty}$-stable  $\Zp$-lattices $L \subset V $.
\end{lemma}
\begin{proof}
This is the analogue of \cite[Lem. 2.1.15]{Kis06}, and the proof follows the same argument, by using our Prop. \ref{prop42} and Prop. \ref{prop43}.
\end{proof}

The following lemma will be used later.
\begin{lemma}\label{4211}
\begin{enumerate}
\item  For $\wt\gm \in \Mod_{\gs_\bbk}^\varphi$, there is an $G_{\bbkinfty}$-equivariant isomorphism 
$$T_{\gs_\bbk}(\wt\gm) \simeq (\wt\gm \otimes_{\gs_\bbk} W(\mathbb{C}^\flat))^{\varphi=1}.$$
\item For $ \gm \in \Mod_{\gs }^\varphi$, there is an $G_{\bbkinfty}$-equivariant isomorphism 
$$T_{\gs }( \gm) \simeq ( \gm \otimes_\gs W({C}^\flat))^{\varphi=1}.$$
\end{enumerate} 
\end{lemma}
\begin{proof}
Item (1) follows from \cite[Lem. 2.1.4, Lem. 2.3.1]{GLAMJ}. Item (2) then follows. Note that the $\varphi$-equivariant isomorphism $W({C}^\flat) \simeq W(\mathbb{C}^\flat)$ induced by $i_\varphi$ is $G_\bbk$-equivariant.
\end{proof}

\subsection{Semi-stable representations and Breuil-Kisin modules} \label{new43}
Given $\gm \in \Mod_\gs^\varphi$ or $M \in \Mod_{\OE}^\varphi$, then the functor in \eqref{eqmmk} implies that $\gm$ or $M$ are fixed by $G_{\bbkinfty}$ inside $T\otimes_\zp \O_{\wh{\E}^\ur}$. In this subsection, we show that when $\gm$ and $M$ come from semi-stable representations, then they are furthermore fixed by $G_\Kinfty$ (inside some suitable space). This is important for  development   in the next subsection.

We first very briefly recall some notations of overconvergent elements, they are well-known and we only use some very elementary properties of them, cf. e.g. \cite{Ber02} for more details.

\begin{notation}
Denote 
$$\wta=W(C^\flat), \quad \wtb=W(C^\flat)[1/p], $$ 
and let $\wta^\dagger, \wtb^\dagger$ be the sub-ring of overconvergent elements, cf. \cite[\S 1.3]{Ber02}; indeed, $\wtb^\dagger =\wta^\dagger[1/p]$ is even a sub-field of the field $\wtb$.
Recall that  (cf. \cite[\S 2.1]{Ber02})
$$\wtb^\dagger = \cup_{n\geq 0}\wtb_{[r_n, +\infty]},$$
where $r_n := (p-1)p^{n-1}$ and $\wtb_{[r_n, +\infty]}$ is defined to be the subring of overconvergent elements with suitable range of overconvergence. 
By \cite[\S 2.2]{Ber02}, for each $n\geq 0$, there is a $G_K$-equivariant embedding
$$\iota_n: \wtb_{[r_n, +\infty]} \into \bdr^{\nabla  +}.$$
\end{notation}

 \begin{defn}\hfill
 \begin{enumerate}
 \item Let $\Mod_{\wta}^{\varphi, G_K}$ be the category of the following data: 
\begin{enumerate}
\item a finite free $\wta$-module $M$ ;
\item a  $\varphi_{\wta}$-semi-linear  and bijective  map $\varphi: M\to M$;
\item a $\varphi$-commuting $G_K$-action on $M$, which is semi-linear with respect to the $G_K$-action on $\wta$.
\end{enumerate}  
 \item Let $\Mod_{\wta^\dagger}^{\varphi, G_K}$ be similarly defined.
 \end{enumerate} 
 \end{defn}

 \begin{lemma}\label{thmwfr} 
 There are equivalences of categories
$$\Rep_{\Zp}(G_K) \to \Mod_{\wta^\dagger}^{\varphi, G_K} \to \Mod_{\wta}^{\varphi, G_K},$$
where the first functor sends $T$ to $ T\otimes_\zp \wta^\dagger$, and the second functor is defined by   base-change.
 \end{lemma} 
\begin{proof}
This is well-known; essentially it is because $C^\flat$ is algebraically closed and is a perfectoid field, cf. e.g., \cite[Thm. 8.5.3]{KL15}.
\end{proof}  

\begin{prop}\label{prop46}
Recall that in Thm. \ref{323}, we have constructed a fully faithful functor 
$$(\ast): \MFnablawageq \to \Modgsnabla\otimes_\Zp \Qp.$$
Suppose $D \in \MFnablawageq$ maps to $(\gm, \varphi, N, \nabla)$ (up to isogeny), then there is a canonical $G_\bbkinfty$-equivariant isomorphism 
$$T_\gs(\gm)\otimes_{\Zp}\Qp \simeq V_\st(D)|_{G_\bbkinfty}$$
\end{prop}
\begin{proof}
$T_\gs(\gm) =T_{\gs_\bbk}(\gm\otimes_\gs \gs_\bbk)$, hence the propostion follows from \cite[Prop. 2.1.5]{Kis06}.
\end{proof}

\begin{prop}\label{p434}
Let  $D$ and $\gm$ be as in Prop. \ref{prop46}, namely, $\gm$ (which is defined up to isogeny) comes from a semi-stable representation $V:=V_\st(D)$. Prop. \ref{prop46} and Lem. \ref{4211} induces a $\varphi$-equivariant isomorphism 
\begin{equation}\label{eqp434}
\gm \otimes_\gs \wtb  \simeq V\otimes_\qp \wtb;
\end{equation}
and the $G_K$-action on the right hand side induces a $G_K$-action on the left hand side.
Then $\gm$ is fixed by $G_{\Kinfty}$ under this action.
\end{prop}
\begin{proof}
We have the following  $\bst$-linear isomorphisms
\begin{eqnarray}
\label{4351} V\otimes_\qp \bst &\simeq & D\otimes_{\ko}\bst \\
 \label{4352} &\simeq & (D\otimes_\ko S[1/p]) \otimes_{S[1/p] } \bst \\
 \label{4353} &\simeq & (\gm\otimes_{\varphi, \gs}  S[1/p]) \otimes_{S[1/p] } \bst
\end{eqnarray} 
Here \eqref{4351} is $G_K$-equivariant with trivial $G_K$-action on $D$. In \eqref{4352}, $D\otimes_\ko S[1/p]$ is fixed by $G_\Kinfty$ since $S \subset (\bst)^{G_\Kinfty}$. Thus in \eqref{4353} -- which holds by Prop. \ref{p325} -- the subset $(\gm\otimes_{\varphi, \gs}  S[1/p])$ is also fixed by $G_\Kinfty$. 
Thus $\varphi^\ast\gm$ is fixed by $G_\Kinfty$ in the following base change:
\begin{equation} \label{e432a}
V\otimes_\qp \bdr \simeq \varphi^\ast\gm\otimes_\gs \bdr.
\end{equation} 

Lem. \ref{thmwfr}  and Eqn. \eqref{eqp434} implies that we have a  $\varphi$-equivariant isomorphism (we can use $\varphi^\ast \gm$ here because $\varphi$ is bijective on $\wtb$ and $\wtb^\dagger$)
\begin{equation*}
 V\otimes_\Qp \wtb^\dagger \simeq \varphi^\ast \gm\otimes_\gs \wtb^\dagger.
\end{equation*}
Since $\wtb^\dagger$ is a field, there exists some $n \gg 0$ such that
\begin{equation*}
  V\otimes_\Qp \wtb_{[r_n, +\infty]} \simeq \varphi^\ast \gm\otimes_\gs \wtb_{[r_n, +\infty]}.
\end{equation*}
Apply the $G_K$-equivariant base changes $\iota_n: \wtb_{[r_n, +\infty]} \into \bdr^{\nabla  +}$ and  $\bdr^{\nabla  +} \into \bdr$, we recover \eqref{e432a}; thus $\varphi^\ast \gm$ and hence $\gm$ is fixed by $G_\Kinfty$.
\end{proof}

\subsection{Breuil-Kisin  $ G_K$-modules}\label{sub43}
In this subsection, we study integral semi-stable representations of $G_K$. The following lemma will be repeatedly used.

\begin{lemma}\label{lemsubgen}
The two subgroups $G_\bbk$ and $G_{\Kinfty}$ generate $G_K$.
\end{lemma}
\begin{proof}
This is noted in the proof of \cite[Prop. 3.8]{BT08}; indeed $\bbk \cap \Kinfty=K$ as $\Kinfty$ is totally ramified.
\end{proof}

Now we introduce the Breuil-Kisin  $ G_K$-modules.
Let $\mathfrak{m}_{\O_C^\flat}$ be the maximal ideal of $\O_C^\flat$, and let $W(\mathfrak{m}_{\O_C^\flat})$ be the ideal of Witt vectors.

\begin{defn}\label{defwr}
Let $\textnormal{Mod}_{\gs, W(\O_C^\flat)}^{\varphi, G_K}$
be the category consisting of data which we call the (effective) \emph{Breuil-Kisin  $ G_K$-modules}:
\begin{enumerate}
\item $(\mathfrak{M}, \varphi_\mathfrak{M})\in \Mod_{\gs}^{\varphi}$;
\item  $G_K$ is a continuous $W(\O_C^\flat)$-semi-linear   $G_K$-action on $\hM:= W(\O_C^\flat)
\otimes_{ \gs} \mathfrak{M}$;
\item $G_K$ commutes  with $\varphi_{\widehat{\mathfrak{M}}}$ on $\widehat{\mathfrak{M}}$;
\item $\gm \subset \hM^{G_\Kinfty}$ via the embedding $\gm \hookrightarrow \hM$;
\item  $\gm/u\gm \subset (\hM/W(\mathfrak{m}_{\O_C^\flat})\hM)^{G_K}$ via the embedding $\gm/u\gm \hookrightarrow \hM/W(\mathfrak{m}_{\O_C^\flat})\hM$.
\end{enumerate}
  \end{defn}

\begin{theorem}\label{thmgao}
Starting from $\bbk$, let $\textnormal{Mod}_{\gs_\bbk, W(\O_\mathbb{C}^\flat)}^{\varphi, G_\bbk}$
be the category    defined similarly as Def. \ref{defwr}: namely, one changes the notations $$``\gs, \quad \O_{C}^\flat, \quad G_K, \quad G_\Kinfty"$$ there to $$\gs_\bbk, \quad \O_\mathbb{C}^\flat, \quad G_\bbk, \quad G_\bbkinfty.$$
 Then the functor
$$(\gm, \hM) \mapsto (\gm \otimes W(\mathbb{C}^\flat))^{\varphi=1}$$
induces an equivalence of categories
$$\textnormal{Mod}_{\gs_\bbk, W(\O_\mathbb{C}^\flat)}^{\varphi, G_\bbk} \xrightarrow{\simeq }  
\Rep_{\Zp}^{\st,\geq 0}(G_\bbk).$$
\end{theorem}  
  \begin{proof}
  This follows from the main theorem of \cite{Gaolp}.
  \end{proof}
  
The following is our main theorem on integral $p$-adic Hodge theory in the imperfect residue field case.  
  
  \begin{theorem}\label{thm411}
  There is a fully faithful functor 
  \begin{equation*}  \Rep_{\Zp}^{\st,\geq 0}(G_K)
\xrightarrow{ }  
\Mod_{\gs, W(\O_C^\flat)}^{\varphi, G_K},
\end{equation*}
such that if $T$ maps to  $(\gm, \hM)$, then there is a $G_K$-equivariant isomorphism 
$$T \simeq   (\hM \otimes_{W(\O_{C}^\flat)} W({C}^\flat))^{\varphi=1}.$$
  \end{theorem}
  \begin{proof}
  Given $T\in \Rep_{\Zp}^{\st,\geq 0}(G_K)$, Prop. \ref{prop46} and Lem. \ref{45} imply that there is a unique $\gm \in \Mod_\gs^\varphi$ corresponding to $T|_{G_{\bbkinfty}}$. The assignment $T \mapsto \gm$ is functorial because $T_\gs$ is fully faithful.
  
Now we want to construct a natural $G_K$-action on $\gm \otimes_\gs W(\O_C^\flat)$. 
First, consider $T$ as an object in $\Rep_{\Zp}^{\st,\geq 0}(G_\bbk)$; by the construction in  \S \ref{423c} and Thm. \ref{thmgao}, 
$\gm_\bbk:=\gm \otimes_\gs \gs_\bbk$ is the corresponding object in $\Mod_{\gs_\bbk}^\varphi$ and 
\begin{equation} \label{equibbk}
\gm_\bbk  \otimes_{\gs_\bbk} W(\O_\mathbb{C}^\flat) \into T \otimes_\Zp W(\mathbb{C}^\flat) \text{ is } G_\bbk\text{-stable}.
\end{equation} 
  Consider the image of  the following $\varphi$-equivariant inclusions (where the second isomorphism follows from Lem. \ref{4211} and Lem. \ref{thmwfr}):
\begin{equation}\label{eqmcf}
\gm \otimes_\gs W(\O_C^\flat) \subset \gm\otimes_\gs W(C^\flat) \simeq T \otimes_\Zp W(C^\flat).
\end{equation} 
Using \eqref{equibbk} and the $G_\bbk$-equivariant isomorphism $W(\O_C^\flat) \simeq W(\O_\mathbb{C}^\flat)$  induced by $i_\varphi$ in \S \ref{sec21}, we see that the image of \eqref{eqmcf} is $G_{\bbk}$-stable; 
it is also $G_{\Kinfty}$-stable   because Prop. \ref{p434} implies that $G_{\Kinfty}$ acts trivially on $\gm$. 
Thus the image  is $G_K$-stable by Lem. \ref{lemsubgen}, and this induces the desired $G_K$-action on $\gm \otimes W(\O_C^\flat)$.  The construction of $G_K$-action is functorial with respect to $T$ by Lem. \ref{thmwfr}.
To show that $\gm/u\gm$ is fixed by $G_K$,  simply  consider   $G_{\bbk}$- and   $G_{\Kinfty}$-actions again, and use Lem. \ref{lemsubgen}.

The paragraph above gives the desired functor. 
It is faithful because $T_\gs$ is fully faithful.
To show fullness, let $T_1, T_2 \in \Rep_{\Zp}^{\st,\geq 0}(G_K) $, and let $\hm_1, \hm_2 \in \textnormal{Mod}_{\gs, W(\O_C^\flat)}^{\varphi, G_K}$ be the corresponding objects. A morphism $\hm_1 \to \hm_2$ induces a morphism $$\hm_1 \otimes_{W(\O_C^\flat)} W(C^\flat) \to \hm_2 \otimes_{W(\O_C^\flat)} W(C^\flat)$$ in $\textnormal{Mod}_{\wta}^{\varphi, G_K}$, and hence induces the desired morphism $T_1 \to T_2$ by Lem. \ref{thmwfr}. 
  \end{proof}
  
  \begin{rem} \label{rem437}
  \begin{enumerate}
  \item We do not know if the functor in Thm. \ref{thm411} is essentially surjective. Indeed, take an object $(\gm, \hM) \in \Mod_{\gs, W(\O_C^\flat)}^{\varphi, G_K}$, and let $T= (\hM \otimes_{W(\O_{C}^\flat)} W({C}^\flat))^{\varphi=1}.$ Then base change along $\gs \to \gs_\bbk$ (cf. \eqref{equibbk}) gives an object in $\textnormal{Mod}_{\gs_\bbk, W(\O_\mathbb{C}^\flat)}^{\varphi, G_\bbk}$, hence Thm. \ref{thmgao} implies that $T|_{G_\bbk}$ is semi-stable. Thus by Thm. \ref{thmmorita}(\ref{item4morita}), $T$ is \emph{potentially} semi-stable. However, we do not know if $T$ is semi-stable.
  
  \item  Let   $T\in \Rep_{\Zp}^{\st,\geq 0}(G_K)$, and let $(\mathfrak{M}, \varphi_\mathfrak{M}, G_K)$ be the corresponding Breuil-Kisin  $ G_K$-module. It might be tempting to add some connection operator $\nabla$, say   on $\gm/u\gm$ or just on $\gm/u\gm\otimes_\Zp \Qp$. The later map
$$ \nabla: \gm/u\gm\otimes_\Zp \Qp \to (\gm/u\gm \otimes_\Zp \Qp)\otimes_\ko \wh{\Omega}_\ko $$  
  certainly   exists as it corresponds to the connection on the Fontaine module of $T[1/p]$; but then it loses integrality information. 
 Furthermore, it is not clear if the above map induces
  $$\nabla: \gm/u\gm \to \gm/u\gm\otimes_\oko \wh{\Omega}_\oko.$$ 
  \end{enumerate} 
\end{rem}

\section{Modules over $\gs$ and $S$}\label{sec5}
In this section, we review some results about $\varphi$-modules over $\gs$ and $S$. Our exposition here is more general than what is needed in \S \ref{secpdiv}; we use this opportunity to  discuss various related results in the literature.

\subsection{On $\varphi$-modules over $\gs$ and $S$}
Recall in Notation \ref{n324}, we defined $S$. Explicitly, 
$$ S =\{ x=\sum_{i \geq 0} a_i \frac{E(u)^i}{i!} \mid  a_i \in \O_{K_0}[u], \lim_{i \to \infty} v_p(a_i)=+\infty  \} \subset K_0[[u]].$$
  Let $\Fil^r S$ be the ideal topologically generated by $(\frac{E(u)^i}{i!}, i \geq r)$. 
  One can check $S/\Fil^1 S \simeq \O_K$ via $u \mapsto \pi$. 
  Let $I_u \subset S$ be the kernel of the $\O_{K_0}$-linear map $f_0: S \to \O_{K_0}$ where $u \mapsto 0$, then $S/I_u \simeq \O_{K_0}$. 
  Let $\varphi$ be the Frobenius operator on $S$ extending $\varphi$ on $\O_{K_0}$ such that $\varphi(u)=u^p$. 
  For $1 \leq r \leq p-1$, one can check $\varphi(\Fil^r S) \subset p^rS$, and hence one can define 
$$\varphi_r=\frac{\varphi}{p^r}: \Fil^r S \to S.$$

\begin{defn}
Suppose $1 \leq r \leq p-1$.
Let $\Mod_S^{\varphi, r}$ be the category of the following data:
\begin{enumerate}
\item $M$ is a finite free $S$-module;
\item $\Fil^r M \subset M$ is a sub-$S$-module which contains $\Fil^r S \cdot M$;
\item $\varphi_r:\Fil^r M \to M$ is a $\varphi_S$-semi-linear map such that the linearization map $\varphi^\ast \Fil^r M \xrightarrow{1\otimes \varphi_r} M$ is surjective.
\end{enumerate}
\end{defn}

\begin{construction}\label{cons52}
Let $\Mod_\gs^{\varphi, r}$ be the sub-category of $\Mod_\gs^\varphi$ consisting of objects of $E(u)$-height $\leq r$.
There is a natural functor 
$$\cm_\gs^r: \Mod_\gs^{\varphi, r} \to \Mod_S^{\varphi, r}$$
 defined as follows. Given $(\gm, \varphi) \in \Mod_\gs^{\varphi, r}$, let 
 $$M=\cm_\gs^r(\gm):=S\otimes_{\varphi, \gs} \gm,$$
  let 
$$\Fil^r M= \{ m \in M, (1\otimes \varphi)(m) \in \Fil^rS \otimes_\gs \gm \},$$
and let $\varphi_r: \Fil^r M \to M$ be the composite:
$$\Fil^r M \xrightarrow{1\otimes \varphi}   \Fil^rS \otimes_\gs \gm \xrightarrow{\varphi_r\otimes 1} S\otimes_{\varphi, \gs} \gm=M. $$
\end{construction}

The following theorem and remark collect various results and arguments in \cite{CL09, Kim15, Gao17}.

\begin{theorem}\label{thm53}
Suppose $1\leq r \leq p-1$.
\begin{enumerate}
\item When $p>2$ and $1\leq r\leq p-2$, the functor $\cm_\gs^r$ is an equivalence  of categories.
\item For any $p$ and $1\leq r \leq p-1$, $\cm_\gs^r$ is essentially surjective.
\item For any $p$ and $1\leq r \leq p-1$, $\cm_\gs^r$ induces an equivalence of isogeny categories 
$$\Mod_\gs^{\varphi, r}\otimes_\Zp \Qp \to \Mod_S^{\varphi, r}\otimes_\Zp \Qp.$$
\end{enumerate}
\end{theorem}
\begin{proof} 
Item (1) is proved in \cite[Thm. 2.2.1]{CL09}, under the assumption that $K$ has perfect residue field; but the  proof is still valid for general $K$.
Items (2) and (3) are proved in \cite[Prop. 6.6]{Kim15}, which even works in the \emph{relative} case (cf. Rem. \ref{r525}). The proof in \cite[Prop. 6.6]{Kim15} uses similar ideas as in \cite[Thm. 2.2.1]{CL09}, but carries out more refined analysis. Actually, the statement in  \cite[Prop. 6.6]{Kim15} only deals with the case when $r=1$, but the proof works for  general $r$ as stated here: indeed, the only subtle case (in \cite[Prop. 6.6]{Kim15} and in here) is when $p=2$ (and hence $r=1$ then).
\end{proof}

\begin{rem} \label{rembtwrong}
In this remark, we first point out a gap in \cite{BT08} related with Thm. \ref{thm53} above. 
We then discuss   some related results in \cite{Gao17}: these discussions are not  used in the current paper, but we hope it serves to clarify the relation between  various results.
\begin{enumerate}

\item  In \cite[Prop. 6.8]{BT08}, it is claimed that $\cm_\gs^1$ is fully faithful (including $p=2$); unfortunately there is a gap in the argument. 
Using notations in \emph{loc. cit.}, the matrix $\Phi_1$(similarly $\Phi_2$) is defined for the \emph{linearization} of the Frobenius operator on $\gm_1$, hence instead of the   claim that 
$$``\Phi_1 \in \GL_{r_1}(\gs[E(u)^{-1}]) \text{ and  } E(u)\cdot \Phi_1^{-1} \in \mathrm{M}_{r_1}(\gs)",$$ 
we only have
$$\Phi_1 \in \GL_{r_1}(\gs[\varphi(E(u))^{-1}]) \text{ and  } \varphi(E(u))\cdot \Phi_1^{-1} \in \mathrm{M}_{r_1}(\gs).$$ 
Thus \cite[Lem. 6.9]{BT08} becomes irrelevant.
 Furthermore, the ``modified" statement   of \emph{loc. cit.} :
 \begin{itemize}
 \item   `` if $f\in S$ and $\varphi(E(u))^N \cdot f \in \gs+p^NS$ for all $N \geq 0$, then $f\in \gs$"
 \end{itemize}
  is false; e.g., $f=\frac{E(u)^p}{p}$ is a  counter-example.
 Indeed, we  \emph{do not know}  if $\cm_\gs^1$ is fully faithful when $p=2$; see also next item.

 \item As part of \cite[Thm. 2.5.6]{Gao17}, it is shown that 
\begin{itemize}
\item (Statement A): $\cm_\gs^{p-1}$ (including $p=2$) is fully faithful when \emph{restricted to the unipotent sub-categories}.
\end{itemize}
For the notion of \emph{unipotency}, we refer the reader to \cite[Def. 2.1.1, Def. 2.2.2]{Gao17}, and also Part 0 in the proof of  \cite[Thm. 2.5.6]{Gao17}.
 In fact, a \emph{stronger} result is proved there:  
\begin{itemize}
\item (Statement B): the ``mod $p$ functor" of $\cm_\gs^{p-1}$, denoted as 
$$``\cm_{\gs_1}: \Mod_{\gs_1}^\varphi \to \Mod_{S_1}^{\varphi}"$$ 
in \emph{loc. cit.} is an equivalence of categories when \emph{restricted to the unipotent sub-categories}.
\end{itemize}
  In fact, one can \emph{directly} prove the weaker (Statement A) using similar strategy as in (Statement B). However, if one examines the proof, it seems very unlikely that one can remove unipotency condition there.

\item The other part of \cite[Thm. 2.5.6]{Gao17} shows that $\cm_\gs^{p-1}$ is essentially surjective (hence an equivalence)  when \emph{restricted to the unipotent sub-categories}; this can be regarded as supplement to results listed in Thm. \ref{thm53}.

The proof of essentially surjectivity in \cite[Thm. 2.5.6]{Gao17} uses similar strategy as \cite[Thm. 2.2.1]{CL09} and \cite[Prop. 6.6]{Kim15}, but the result is \emph{weaker} than  Thm. \ref{thm53}(2) stated here. 
Indeed, in the proof of \cite[Thm. 2.5.6]{Gao17}, we make use of a ring $``\Sigma" \subset S$, hence the iteration procedure actually becomes harder to converge: this explains why in \emph{loc. cit.} we did not obtain Thm. \ref{thm53}(2).
\end{enumerate}
\end{rem}

\subsection{Modules with connections}
\begin{defn} \label{d521}
Let $\mathrm{MF}^{\mathrm{BT}}_S(\varphi, \nabla)$ be the category of the following data:
\begin{enumerate}
\item $(M, \Fil^1 M, \varphi_1) \in \Mod_S^{\varphi, 1}$;
\item $\nabla: M/I_uM \to M/I_uM\otimes_{\O_{K_0}} \wh{\Omega}_\oko$ is a  quasi-nilpotent integrable connection which commutes with $\varphi$-action on $M/I_uM$.
\end{enumerate}
\end{defn}

\begin{defn} \label{def57}
Let $\mathrm{MF}^{\mathrm{BT}}_\gs(\varphi, \nabla)$ be the category of the following data (which are called \emph{minuscule  Breuil-Kisin modules with connections}): 
\begin{enumerate}
\item $(\gm, \varphi) \in \Mod_{\gs}^{\varphi, 1}$;
\item $\nabla: \varphi^\ast(\gm/u\gm) \to \varphi^\ast(\gm/u\gm) \otimes_{\O_{K_0}} \wh{\Omega}_\oko$ is a  quasi-nilpotent integrable connection which commutes with $\varphi$-action on $\varphi^\ast(\gm/u\gm)$, where $$\varphi^\ast(\gm/u\gm)= \gm/u\gm \otimes_{\varphi, \oko}\oko.$$
\end{enumerate}
\end{defn}

\begin{rem} \label{r523}
\begin{enumerate}
\item In Def. \ref{def57}, we need to define $\nabla$ over the ``Frobenius twist" $\varphi^\ast(\gm/u\gm)$ so that the it is \emph{compatible} with the  $\nabla$ in Def. \ref{d521} under the functor \eqref{eqf52} below. Our definition is compatible with that in \cite[Def. 6.1]{Kim15}.
Note however that in \cite[Def. 4.16]{BT08}, the $\nabla$-operator is defined over $\gm/u\gm$; however, it seems likely to be a  \emph{wrong} definition, see next item.

\item As noted in Rem. \ref{remdimp}, for a Fontaine module $D$, there is a $K_0$-linear  isomorphism $\varphi^\ast(D) \simeq D$; hence to define $\nabla$ over $D$ is \emph{equivalent} to do so over  $\varphi^\ast(D)$. Indeed, this is why we do not need ``Frobenius twist" in the $\nabla$-operators in Def. \ref{deffilnmod} and all the various categories in \S \ref{sec3}, as all of them are defined on the $K_0$-level.
However, on the \emph{integral} level, we do no expect $\varphi^\ast(\gm/u\gm) \simeq \gm/u\gm$ in general.
\end{enumerate}
\end{rem}

The functor in Construction \ref{cons52} obviously induces a functor 
\begin{equation} \label{eqf52}
\BT_\gs^S: \mathrm{MF}^{\mathrm{BT}}_\gs(\varphi, \nabla) \to \mathrm{MF}^{\mathrm{BT}}_S(\varphi, \nabla).
\end{equation}
 The following corollary follows obviously from Thm. \ref{thm53}.
 
\begin{cor}\label{c524}
Consider the functor $\BT_\gs^S$. 
It is an equivalence of categories when $p>2$. It is  essentially surjective when $p=2$. For any $p$, it induces an equivalence of the isogeny categories.
\end{cor}

\begin{rem}\label{r525}
Let $R$ be as in Assumption \ref{assr}, and suppose it satisfies Assumption  (i) there.
Then \cite[Prop. 6.6]{Kim15} proves the relative version of Thm. \ref{thm53} and hence Cor. \ref{c524}. 
\end{rem}

\section{Breuil-Kisin modules and $p$-divisible groups}\label{secpdiv}

In this section, we study $p$-divisible groups over $\ok$. The first two subsections \S \ref{s61}, \S \ref{s62} contain results established by Brinon-Trihan (except Thm. \ref{thmtp}), on the relationship between $p$-divisible groups, $S$-modules and crystalline representations.
In \S \ref{sub63}, we classify $p$-divisible groups by Breuil-Kisin modules with connections; the case when $p>2$ is already known, the case $p=2$ requires new ideas and our approach is informed by the integral theory  developed in \S \ref{sec4}. Finally in \S \ref{sub64}, we classify finite flat group schemes over $\ok$.
 
\subsection{$p$-divisible groups and  $S$-modules}\label{s61}
Let $ \mathrm{BT}(\O_K)$ be the category of $p$-divisible groups over $\ok$.
Let $G \in \mathrm{BT}(\O_K)$; it corresponds to a compatible system $(G_n)_{n>0}$ where $G_n$ is a Barsotti-Tate group  over $\ok/p^n\ok$ for each $n$. Let $\bfD^\ast(G_n)$ be the \emph{co-variant} (cf. Convention \ref{subsubco}) Dieudonn\'e crystals, and consider its evaluation on the thickening $S \to \ok/p^n\ok$. Define
$$ \mathbf{M}(G)=\bfD^\ast(G)(S\to\ok):=\projlim_{n>0}\bfD^\ast(G_n)(S\to \ok/p^n\ok). $$

\begin{thm} \label{thm61}
The above construction gives rise to a functor
$$\mathbf{M}: \mathrm{BT}(\O_K) \to \mathrm{MF}^{\mathrm{BT}}_S(\varphi, \nabla).$$
The functor is an equivalence if $p>2$; it induces an equivalence of isogeny categories if $p=2$.
\end{thm}
\begin{proof}
This is \cite[Prop. 5.9, Thm. 5.10]{BT08}.  
\end{proof}

\begin{prop} 
Let $G\in \mathrm{BT}(\O_K)$, the functor in Thm. \ref{thm61} induces a   $(\varphi, \nabla)$-equivariant isomorphism  of  $\oko$-modules:
\begin{equation}\label{eq611}
\bfD^\ast(G\otimes_\ok k)(\oko)=\mathbf{M}(G)\otimes_S \oko=\bfm(G)/I_u\bfm(G).
\end{equation}
Furthermore, we have a $\varphi$-equivariant isomorphism of $S[1/p]$-modules:
\begin{equation}\label{eq612}
\mathbf{M}(G)[1/p] \simeq \bfD^\ast(G\otimes_\ok k)(\oko)\otimes_\oko S[1/p].
\end{equation}
\end{prop}
\begin{proof}
\eqref{eq611} is extracted from the proof of \cite[Prop. 5.9]{BT08}, 
and \eqref{eq612} is \cite[Lem. 6.1]{BT08}.
\end{proof}

\begin{rem}\label{r613}
Let $R$ be as in Assumption \ref{assr}, and suppose it satisfies Assumption  (ii) there.
Then \cite[Thm. 3.17]{Kim15} proves the relative version of Thm. \ref{thm61}; it implies the relative version of \eqref{eq611}.
The relative version of \eqref{eq612} also follows from similar argument as in \cite[Lem. 6.1]{BT08}.
\end{rem}

\subsection{$p$-divisible groups and crystalline representations}\label{s62}
Let $\Rep_{\Qp}^{\cris, 1}(G_K)$ (resp.  $\Rep_{\Zp}^{\cris, 1}(G_K)$) be the category of crystalline (resp. integral crystalline) representations of $G_K$ with Hodge-Tate weights in $\{0, 1\}$.
In this subsection, we show that $\Rep_{\Zp}^{\cris, 1}(G_K)$ is equivalent to $\BT(\ok)$.

\begin{prop} \label{pr63}
\cite[Cor. 6.4]{BT08}
Let $G \in \BT(\ok)$, and let $T_p(G)$ be its \emph{co-variant} Tate module. Then $T_p(G) \in \Rep_{\Zp}^{\cris, 1}(G_K)$. 
Furthermore, if we let $V_p(G)=T_p(G)[1/p]$, then there is a $(\varphi, \nabla)$-equivariant isomorphism of $K_0$-vector spaces:
$$D_\cris(V_p(G))=\bfD^\ast(G_k)(\oko)[1/p].$$
\end{prop}

\begin{rem}
Let $R$ be   a normal domain (the general assumption in \cite[\S 4, \S5]{Kim15}), and suppose it satisfies Assumption  \ref{assr}(iii), then \cite[Thm. 5.2]{Kim15} and particularly \cite[Cor. 5.3]{Kim15} proves the relative version of Prop. \ref{pr63}.
\end{rem}

\begin{theorem} \cite[Thm. 6.10]{BT08} \label{thm610}
The functor $V_p$ induces an equivalence of categories
$$\mathrm{BT}(\O_K)\otimes_\Zp \Qp \to  \Rep_{\Qp}^{\cris, 1}(G_K)$$
\end{theorem}
\begin{proof} 
In the beginning of  the proof of \cite[Thm. 6.10]{BT08}, it says that ``we already know that the functor is fully faithful"; it is not clear to us what the authors mean: it is certainly not obvious that $V_p$ is fully faithful (at that point; also see our exposition in the following). (As a side note, it seems to us the authors are \emph{not} pointing to \cite[Prop. 6.8]{BT08}, whose proof is wrong as we mentioned in Rem. \ref{rembtwrong}(1).)

In any case, we choose to give an exposition of the proof. It particularly serves as a good place to see how to piece together the various module \emph{comparisons}  we have obtained so far.

Let $\mathbf{MF}^{\mathrm{wa}, 1}_{K/K_0}(\varphi, \nabla) \subset \MFnablawa$ be the sub-category consisting of objects with $N=0$ and with Hodge-Tate weights in $\{0, 1\}$.
Consider the composite of functors in the following diagram:

\begin{equation}\label{eq610}
\begin{tikzcd}
\mathrm{BT}(\O_K)\otimes_\Zp \Qp                                                &                                                                                             & {\mathrm{MF}^{\mathrm{BT}}_S(\varphi, \nabla) \otimes_\Zp \Qp} \arrow[ll, "{\simeq, \text{Thm. \ref{thm61}}}"']  \\
{\Rep_{\Qp}^{\cris,  1}(G_K) } \arrow[rd, "{\simeq, \text{Thm. \ref{thmCF}}}"'] &                                                                                             & {\mathrm{MF}^{\mathrm{BT}}_\gs(\varphi, \nabla) \otimes_\Zp \Qp} \arrow[u, "{\simeq, \text{Thm. \ref{thm53}}}"'] \\
                                                                                & {\mathbf{MF}^{\mathrm{wa}, 1}_{K/K_0}(\varphi, \nabla)} \arrow[ru, "\text{Thm. \ref{323}}"', hook] &                                                                                                                 
\end{tikzcd}
\end{equation}
The only caveat in the diagram is the fully faithful functor
$$ \mathbf{MF}^{\mathrm{wa}, 1}_{K/K_0}(\varphi, \nabla) \to  \mathrm{MF}^{\mathrm{BT}}_\gs(\varphi, \nabla) \otimes_\Zp \Qp.$$
Indeed,  the functor labelled as $(\ast)$ in  Thm. \ref{323}  implies there is a fully faithful functor
$$ \mathbf{MF}^{\mathrm{wa}, 1}_{K/K_0}(\varphi, \nabla) \to  \Mod_{\gs}^1(\varphi, \nabla)\otimes_\Zp \Qp,$$
where $\Mod_{\gs}^1(\varphi, \nabla)$ is the sub-category of $\Mod_{\gs}(\varphi, N, \nabla)$ consisting of modules with $E(u)$-height $\leq 1$ and $N=0$. 
Note that the $\nabla$  for $\Mod_{\gs}^1(\varphi, \nabla)$ in Def. \ref{def39} is defined on the $K_0$-level, whereas $\nabla$ for $\mathrm{MF}^{\mathrm{BT}}_\gs(\varphi, \nabla)$ in Def. \ref{def57} is defined on the ``Frobenius-twisted" $\oko$-level. Nonetheless, we have an equivalence of isogeny categories: 
\begin{equation}
\mathrm{MF}^{\mathrm{BT}}_\gs(\varphi, \nabla) \otimes_\Zp \Qp \simeq  \Mod_{\gs}^1(\varphi, \nabla)\otimes_\Zp \Qp,
\end{equation}
by the discussions  in  Rem. \ref{r523}.

To prove the theorem, it suffices to show that $V_p$ is a  quasi-inverse of the composite in diagram \eqref{eq610}, and hence completes the   diagram. The argument in the following is the same as in \cite[Thm. 6.10]{BT08}.

Let $V\in \Rep_{\Qp}^{\cris,  1}(G_K)$, and suppose it maps to $D, \gm, M, G$ in the corresponding categories using the functors above (where $\gm, M, G$ are constructed up to isogeny). 
\begin{equation}\label{eq622}
\begin{tikzcd}
G                     &                       & M \arrow[ll, maps to] \\
V \arrow[rd, maps to] &                       & \gm \arrow[u, maps to]  \\
                      & D \arrow[ru, maps to] &                        
\end{tikzcd}
\end{equation}

 We have the following comparisons, which are always $(\varphi, \nabla)$-equivariant and compatible with filtrations
\begin{enumerate}
\item $V \otimes_\Qp \bcris \simeq D\otimes_\ko \bcris$, since $V$ is crystalline;
\item  $\gm\otimes_{\varphi, \gs} S \simeq M$, by construction;
\item $D\otimes_{K_0}S[1/p] \simeq M[1/p]$, by Prop. \ref{p325}, where the compatibility with filtrations is discussed in the proof of \cite[Thm. 6.10]{BT08};
\item $M[1/p] \simeq \bfD^\ast(G_k)(\oko)\otimes_\oko S[1/p]$, by \eqref{eq612}.
\item $\bfD^\ast(G_k)(\oko)[1/p] \otimes_\ko \bcris \simeq V_p(G) \otimes_\Qp \bcris$ by Prop. \ref{pr63}.
\end{enumerate}
Combining all these comparisons, we get a  $(\varphi, \nabla)$-equivariant and  filtration-compatible isomorphism
\begin{equation}\label{eqcom}
V \otimes_\Qp \bcris \simeq V_p(G) \otimes_\Qp \bcris.
\end{equation} 
 Note that \emph{a priori}, we do not know if it is $G_K$-equivariant,  essentially because there is no $G_K$-action on $\gm$ or $M$. 
 However, if we equip trivial $G_\Kinfty$-actions on $\gm$ and $M$ (and on $\gs$ and $S$) along all the listed comparisons (1)-(5), then it is easy to see that \eqref{eqcom} is $G_\Kinfty$-equivariant. Taking $\varphi=1, \nabla=0$ in the $\Fil^0$ pieces of \eqref{eqcom}, we see that $V\simeq V_p(G)$ as $G_\Kinfty$-representations; they are furthermore isomorphic as $G_K$-representations by \cite[Prop. 3.8]{BT08}, which says that the restriction  functor    $$\Rep_{\Qp}^{\cris}(G_K) \to \Rep_\Qp(G_\Kinfty)$$ is fully faithful.
\end{proof}

\begin{theorem}\label{thmtp}
The functor $T_p$ from $\mathrm{BT}(\O_K)$ to $\Rep_{\Zp}^{\cris, 1}(G_K)$ is an equivalence.
\end{theorem}
\begin{proof} 
With Thm. \ref{thm610} (the analogue of \cite[Cor. 2.2.6]{Kis06}) established, the proof follows the same argument  in \cite[Thm. 2.2.1]{Liu13}.
 In particular, Tate's isogeny theorem \cite[p. 181, Cor. 1]{Tat67} is still applicable because $\O_K$ is an integrally closed, Noetherian, integral domain with $\mathrm{char } K=0$ as required in \cite[p. 180, Thm. 4]{Tat67}. Also, \cite[Prop. 2.3.1]{Ray74} is still applicable because $\O_K$ is a mixed characteristic discrete valuation ring as required in the beginning of \cite[p. 259, \S 2]{Ray74} and \cite[p. 261, \S 2.3]{Ray74}.
\end{proof}

\begin{rem}\label{r625}
It is far from clear if the \emph{relative} version of Thm.  \ref{thm610} (let alone Thm. \ref{thmtp}) should still hold. Nonetheless, recently, Liu and Moon \cite{LiuMoon} have shown  that these results hold when the ``ramification index" $e$ in Assumption \ref{assr}(1) satisfies $e<p-1$.
\end{rem}



\subsection{$p$-divisible groups and Breuil-Kisin modules}\label{sub63}
The content of Thm. \ref{thm610} completes the diagram  \eqref{eq610}; hence we obtain the following diagram (with top arrow direction reversed) of \emph{equivalences of categories}:
 
\begin{equation}\label{tik2}
 \begin{tikzcd}
\mathrm{BT}(\O_K)\otimes_\Zp \Qp \arrow[d] \arrow[rr] &                                                             & {\mathrm{MF}^{\mathrm{BT}}_S(\varphi, \nabla) \otimes_\Zp \Qp}             \\
{\Rep_{\Qp}^{\cris,  1}(G_K) } \arrow[rd]             &                                                             & {\mathrm{MF}^{\mathrm{BT}}_\gs(\varphi, \nabla) \otimes_\Zp \Qp} \arrow[u] \\
                                                      & {\mathbf{MF}^{\mathrm{wa}, 1}_{K/K_0}(\varphi, \nabla)} \arrow[ru] &                                                                           
\end{tikzcd}
\end{equation}

All the categories in \eqref{tik2} except the bottom one have \emph{integral} avatars, and the content of this subsection is to complete the following diagram of integral categories:
\begin{equation}\label{tik3}
 \begin{tikzcd}
\mathrm{BT}(\O_K) \arrow[d] \arrow[rr]            &  & {\mathrm{MF}^{\mathrm{BT}}_S(\varphi, \nabla)  }             \\
{\Rep_{\Zp}^{\cris,  1}(G_K) } \arrow[rr, dotted] &  & {\mathrm{MF}^{\mathrm{BT}}_\gs(\varphi, \nabla)  } \arrow[u]
\end{tikzcd}
\end{equation}
In particular, we will show that the (to-be-constructed)   dotted arrow is an equivalence of categories: this implies that  $\mathrm{MF}^{\mathrm{BT}}_\gs(\varphi, \nabla) $ is equivalent to $\mathrm{BT}(\O_K).$

\begin{prop}\label{66}
Let $T \in \Rep_{\Zp}^{\cris,  1}(G_K)$, and let $G\in \BT(\ok)$ be the corresponding $p$-divisible group via Thm. \ref{thmtp}.
Let $\gm(T)$ be the Breuil-Kisin module attached to $T$ via  Thm. \ref{thm411}. Then we have a $\varphi$-equivariant isomorphism 
$$\gm(T)\otimes_{\varphi, \gs}S  \simeq \bfm(G).$$
\end{prop}
\begin{proof}
When $K$ has \emph{perfect} residue field, this is \cite[Prop. 2.4]{Kim12}.
In the general case, restricting the representation to $G_\bbk$ implies
$$(\gm(T)\otimes_{ \gs} \gs_\bbk)  \otimes_{\varphi, \gs_\bbk} S_\bbk\simeq \bfm(G)\otimes_S S_\bbk,$$
where $S_\bbk$ is the ``$\bbk$-version" of $S$ defined in \S \ref{sec5}.
Note that the objects in \eqref{eq622} correspond to each other, \emph{up to isogeny}, via the equivalences of categories in \eqref{tik2}, hence we must have
$$\gm(T)\otimes_{\varphi, \gs}S[1/p]  \simeq \bfm(G)\otimes_S S[1/p].$$
We can conclude using the fact that $S_\bbk \cap S[1/p]=S$.
\end{proof}

\begin{prop}\label{67}
There exists a natural functor 
$$\underline{\gm}: \Rep_{\Zp}^{\cris,  1}(G_K) \to \mathrm{MF}^{\mathrm{BT}}_\gs(\varphi, \nabla),$$
which makes \eqref{tik3} commutative, and which after $\otimes_\Zp \Qp$ fits into \eqref{tik2}.
\end{prop}
\begin{proof}
Let $T \in \Rep_{\Zp}^{\cris,  1}(G_K)$, and let $G\in \BT(\ok)$ be the corresponding $p$-divisible group.
By Prop. \ref{66}, $\gm(T)\otimes_{\varphi, \gs}S  \simeq \bfm(G)$, hence the $\nabla$-operator on $\bfm(G)/I_u\bfm(G)$ induces a $\nabla$-operator on $\varphi^\ast(\gm(T)/u\gm(T))$. This gives rise to an object in $\mathrm{MF}^{\mathrm{BT}}_\gs(\varphi, \nabla)$, and obviously defines the desired  functor.
\end{proof}

\begin{theorem}\label{68}
The functor $\underline{\gm}$ is an equivalence, and hence we have equivalences of categories:
$$ \BT(\ok)\xrightarrow{T_p} \Rep_{\Zp}^{\cris,  1}(G_K) \xrightarrow{\underline{\gm}}  \mathrm{MF}^{\mathrm{BT}}_\gs(\varphi, \nabla) $$
\end{theorem} 
\begin{rem} \label{re633}
\begin{enumerate}
\item When $p>2$, the theorem follows by combining Thm. \ref{thm61} (due to Brinon-Trihan) and   Cor. \ref{c524} (which is essentially Thm. \ref{thm53}, and  is due to Caruso-Liu). Our argument here works for any $p$, using only the equivalence up to isogeny in Thm. \ref{thm61} as input; namely, we avoid using Thm. \ref{thm53}.

\item Suppose $R$ satisfies Assumption \ref{assr}(ii) and $p>2$.
As mentiond in Rem. \ref{r525} and Rem. \ref{r613} respectively, \cite[Prop. 6.6]{Kim15} and \cite[Thm. 3.17]{Kim15} prove  the relative version of  Cor. \ref{c524} and Thm. \ref{thm61} respectively; hence   there is an equivalence of categories
$$\BT(R)\xrightarrow{\simeq}  \mathrm{MF}^{\mathrm{BT}}_{\gs_R}(\varphi, \nabla) $$
in this context. Note however the category $\Rep_{\Zp}^{\cris,  1}(G_R)$ is not involved in this equivalence, cf. Rem. \ref{r625}.
\end{enumerate}
\end{rem}

First, recall that Thm. \ref{68} holds in the perfect residue field case.
\begin{theorem} \label{699}
 We have equivalences of categories:
$$ \BT(\O_\bbk)\to \Rep_{\Zp}^{\cris,  1}(G_\bbk) \to \mathrm{MF}^{\mathrm{BT}}_{\gs_\bbk}(\varphi). $$
\end{theorem}
\begin{proof}
This is due to Kisin \cite{Kis06} when $p>2$, and independently to  Kim \cite{Kim12}, Lau \cite{Lau14} and Liu \cite{Liu13} when $p=2$.
\end{proof}

We start with an easy lemma on connections.
\begin{lemma} \label{lem6}
Let
$$F:  (\gm, \nabla_1) \to  (\gm_2, \nabla_1)$$ 
be a morphism in $\mathrm{MF}^{\mathrm{BT}}_\gs(\varphi, \nabla)$.
Suppose the restricted morphism $f: \gm_1 \to \gm_2$ in $\Mod_\gs^\varphi$ is  injective, and induces an isomorphism in $\Mod_\gs^\varphi \otimes_\Zp \Qp$. 
Then $f$ induces an injection
$$\gm_1/u\gm_1  \into \gm_2/u\gm_2, \text{ and hence } \varphi^\ast(\gm_1/u\gm_1) \into   \varphi^\ast(\gm_2/u\gm_2).$$ 
Furthermore, $\nabla_2$ on  $\varphi^\ast(\gm_2/u\gm_2)$ \emph{induces} $\nabla_1$ on $\varphi^\ast(\gm_1/u\gm_1).$
\end{lemma}
\begin{proof} 
The quotient $\gm_2/\gm_1$ is killed by some $p$-power and hence is $u$-torsion free by \cite[Prop. 2.3.2]{Liu07} (which is still applicable in our case). This implies that $\gm_1/u\gm_1  \into \gm_2/u\gm_2.$ Then it is clear that $ \nabla_2$ induces $ \nabla_1$ by looking at the commutative diagram (recall that $\wh{\Omega}_\oko  $ is finite free over $\oko$ )
\begin{equation*}
\begin{tikzcd}
\varphi^\ast(\gm_1/u\gm_1) \arrow[d, hook] \arrow[r, "\nabla_1"] & \varphi^\ast(\gm_1/u\gm_1) \otimes_\oko \wh{\Omega}_\oko \arrow[d, hook] \\
\varphi^\ast(\gm_2/u\gm_2) \arrow[r, "\nabla_2"]                 & \varphi^\ast(\gm_2/u\gm_2) \otimes_\oko \wh{\Omega}_\oko.                
\end{tikzcd}
\end{equation*}
\end{proof}

\begin{proof}[Proof of Thm. \ref{68}]
\textbf{Part 1}: full faithfulness.
$\underline{\gm}$ is faithful because the functor $T_\gs:  \Mod_{\gs }^\varphi  \to \Rep_{\Zp}(G_{\bbkinfty})$ is fully faithful.
 To show fullness, suppose given 
 $$T_1, T_2 \in \Rep_{\Zp}^{\cris,  1}(G_K), \text{ and the corresponding } (\gm_1, \nabla_1), (\gm_2, \nabla_2) \in  \mathrm{MF}^{\mathrm{BT}}_\gs(\varphi, \nabla).$$ 
 Let $\mathfrak f: (\gm_1, \nabla_1) \to (\gm_2, \nabla_2)$ be a morphism. 
 Since $\underline{\gm}$ induces an  equivalence  between the   isogeny categories, $\mathfrak f$ induces a unique $G_K$-equivariant morphism 
 $$f: T_1[1/p] \to T_2[1/p].$$ 
 It suffices to show that $f(T_1) \subset T_2$ (as subsets). One can simply consider the induced morphism $$\mathfrak{f}_\bbk: \gm_1 \otimes_\gs \gs_\bbk \to \gm_2  \otimes_\gs \gs_\bbk, $$
then the $G_\bbk$-equivariant $f$ must satisfies $f(T_1)\subset T_2$ by Thm. \ref{699}.

\textbf{Part 2}: essential surjectivity.  Let $(\gm, \nabla) \in \mathrm{MF}^{\mathrm{BT}}_\gs(\varphi, \nabla)$. Since $\underline{\gm}$ induces an  equivalence  between the   isogeny categories, we can choose some $L \in \Rep_{\Zp}^{\cris,  1}(G_K)$ such that  if we let 
$$\underline{\gm}(L)=(\gn, \nabla_\gn) \in \mathrm{MF}^{\mathrm{BT}}_\gs(\varphi, \nabla)$$ 
be the associated object, then we have a morphism 
$$ (\gm, \nabla) \to  (\gn, \nabla_\gn)$$
such that $\gm \to \gn$ is an injective morphism in $\Mod_\gs^\varphi$ which becomes isomorphism after inverting $p$. By Lem. \ref{lem6}, $\nabla_\gn$ induces $\nabla$.

Let $V=L[1/p]$. By Lem. \ref{4211} and Lem. \ref{thmwfr}, there exist   $\varphi$-equivariant isomorphisms
\begin{equation}\label{eq634}
\gm \otimes_\gs W(\mathbb C^\flat)[1/p] \simeq \gn \otimes_\gs W(\mathbb C^\flat)[1/p]  \simeq V \otimes_\Qp W(\mathbb C^\flat)[1/p].
\end{equation}
 Note that $(\gm\otimes_\gs \gs_\bbk, \varphi\otimes \varphi)$ corresponds to some $p$-divisible group over  $\O_{\bbk}$ and  hence to some integral crystalline representation of $G_\bbk$ via  Thm. \ref{699}; thus by Thm. \ref{thmgao},
\begin{equation}\label{eq635}
\gm \otimes_\gs W(\O_\mathbb{C}^\flat) =(\gm\otimes_\gs \gs_\bbk) \otimes_{  \gs_\bbk  } W(\O_\mathbb{C}^\flat) \text{ is } G_\bbk \text{-stable}.
\end{equation}
 Using the identification $C^\flat =\mathbb{C}^\flat$ via $i_\varphi$ in \S \ref{sec21}, we have a $\varphi$-equivariant isomorphism
 \begin{equation}\label{eq636}
 \gm \otimes_\gs W(  C^\flat)[1/p] \simeq V \otimes_\Qp W(  C^\flat)[1/p],
 \end{equation} 
and hence the  $G_K$-action on $V \otimes_\Qp W(C^\flat)[1/p]$    induces a  $G_K$-action on $\gm \otimes W(C^\flat)[1/p]$. 
Under this action, $\gm \otimes W(\O_C^\flat)$ is $G_\bbk$-stable by \eqref{eq635}.
Prop. \ref{p434} implies $G_{\Kinfty}$ acts trivially on $\gn[1/p]=\gm[1/p]$ as $\gn$ comes from a crystalline representation; hence $\gm \otimes W(\O_C^\flat)$ is also $G_{\Kinfty}$-stable.
Thus  Lem. \ref{lemsubgen} implies that  $\gm \otimes W(\O_C^\flat)$ is $G_K$-stable.
Thus, in particular, $\gm \otimes W(C^\flat)$  is $G_K$-stable, and hence gives rise to a $G_K$-stable $\Zp$-lattice $T \subset L \subset V$. 
Let $\underline{\gm}(T)=(\gm', \nabla')$. Obviously $\gm'=\gm$ (e.g., by Lem. \ref{45}).  
The injective morphism $T \into L$  induces a morphism 
$$ (\gm, \nabla') \to   (\gn, \nabla_\gn), $$
such that $\gm \to \gn$ is injective. Hence Lem. \ref{lem6} implies $\nabla'$ is also induced by $\nabla_\gn$, and hence $\nabla'=\nabla$. Namely, $\underline{\gm}(T)=(\gm,\nabla)$.
\end{proof}


\subsection{Finite flat group schemes}\label{sub64}

Let $\mathrm{FF}(\ok)$ be the category of finite flat group schemes over $\ok$. In this subsection, we classify $\mathrm{FF}(\ok)$ by torsion minuscule Breuil-Kisin modules with connections.

\begin{defn} \label{defff}
\begin{enumerate}
\item Let $\Mod_\gs^{\mathrm{tor}, \varphi, 1}$ be the category consisting of $(\gm, \varphi)$ where $\gm$ is a finite generated $\gs$-module killed by $p^n$ for some $n \geq 1$, and $\varphi: \gm \to \gm$ is a $\varphi_\gs$-semi-linear map such that the $\gs$-span of $\varphi(\gm)$ contains $E(u)\cdot \gm$.

\item Let $\mathrm{MF}^{\mathrm{tor}, \mathrm{BT}}_\gs(\varphi, \nabla)$ be the category consisting of the following data:
\begin{enumerate}
\item $(\gm, \varphi) \in \Mod_\gs^{\mathrm{tor}, \varphi, 1}.$
\item $\nabla: \varphi^\ast(\gm/u\gm) \to \varphi^\ast(\gm/u\gm) \otimes_{\O_{K_0}} \wh{\Omega}_\oko$ is a  quasi-nilpotent integrable connection which commutes with $\varphi$.
\end{enumerate}
\end{enumerate}
\end{defn}

\begin{defn}
For an exact category $\mathscr C$, let $D^b(\mathscr C)$ be the bounded derived category.
\begin{enumerate}
\item  Let $(\Mod_\gs^{\varphi, 1})^\bullet$ be the full subcategory of $D^b(\Mod_\gs^{\varphi, 1})$ consisting of two-term complexes $\gm^\bullet= \gm_1 \into \gm_2$ concentrated in degree 0 and -1, such that the morphism $\gm_1\to \gm_2$ is injective and  $H^0(\gm^\bullet)$ is killed by some $p$-power.

\item  Let $(\BT(\ok))^\bullet$ be the  full subcategory of $D^b(\BT(\ok))$ consisting of isogenies of $p$-divisible groups $G^\bullet= G_1 \to G_2$.
\end{enumerate}
\end{defn}

\begin{prop} \label{615}  
\begin{enumerate}
\item The functor $\gm^\bullet \mapsto H^0(\gm^\bullet)$ induces an equivalence between  $(\Mod_\gs^{\varphi, 1})^\bullet$ and $\Mod_\gs^{\mathrm{tor}, \varphi, 1}$.

\item The functor $G^\bullet \mapsto \Ker(G^\bullet)$ induces an equivalence between  $(\BT(\ok))^\bullet$ and $\mathrm{FF}(\ok)$.
\end{enumerate}
\end{prop} 
\begin{proof}
Item (1) is \cite[Lem. 2.3.4]{Kis06}. Item (2) is proved in the proof of \cite[Thm. 2.3.5]{Kis06}; the argument uses \cite[Thm. 3.1.1]{BBM}, which is applicable for our  $\O_K$. 
\end{proof}

\begin{theorem}\label{fflat}
There is an equivalence of categories
$$ \mathrm{FF}(\O_K)\xrightarrow{\simeq}  \mathrm{MF}^{\mathrm{tor}, \mathrm{BT}}_\gs(\varphi, \nabla).$$
\end{theorem}

\begin{rem}\label{remkim}  
Let $R$ be as in Assumption \ref{assr} and suppose it satisfies Assumption  (ii) there, then \cite[Thm. 9.8]{Kim15} proves the relative version of our Thm. \ref{fflat} when $p>2$. 
The argument there also proves our Thm. \ref{fflat} even when $p=2$, as our Thm. \ref{68} is available when $p=2$. 
Note that the proof in \cite{Kim15} makes crucial use of Vasiu's theory of ``moduli of connections" \cite{Vas13}, which is applicable when $p=2$.
\end{rem}

\begin{proof}[Proof of Thm. \ref{fflat}]
For the reader's convenience, we give a sketch of Kim's argument mentioned in Rem. \ref{remkim}, specialized in the case $R=\ok$. This should only serve as a very brief road-map to Kim's argument; in particular, we do not review Vasiu's theory.

Our Thm. \ref{68} is stronger than the assumption in  \cite[Prop. 9.5]{Kim15}, and our  Prop. \ref{615}(2) is stronger than   \cite[Lem. 9.7]{Kim15}; hence the   argument in \cite[Prop. 9.5]{Kim15} easily shows that we can construct a \emph{fully faithful} functor
$$ {\gm}^\ast :  \mathrm{FF}(\O_K)\to  \mathrm{MF}^{\mathrm{tor}, \mathrm{BT}}_\gs(\varphi, \nabla).$$

The difficulty is to show the above functor is essentially surjective; the proof here follows similar strategy as in \cite[\S 10.4]{Kim15}.
Let $(\gm, \varphi, \nabla) \in \mathrm{MF}^{\mathrm{tor}, \mathrm{BT}}_\gs(\varphi, \nabla)$.
Unlike Prop. \ref{615}(1), it seems very difficult to directly lift $(\gm, \varphi, \nabla)$ to a finite free object \emph{with a connection}.
Nonetheless, Prop. \ref{615}(1) (playing the role of \cite[Lem. 10.10]{Kim15} here) implies that there exists a \emph{finite free} $\gn \in \Mod_\gs^\varphi$ of height $1$ and a $\varphi$-equivariant surjective  morphism $\gn \onto \gm$. 
Let 
$$\cn_0:=\gn \otimes_{\varphi, \gs} \O_{K_0} =\varphi^\ast(\gn/u\gn).$$
One can use the same argument as in \cite[Prop. 10.11]{Kim15} -- which critically uses Vasiu's theory of moduli of connections -- to show that there exists a  faithfully flat  ind-\'etale  map 
$$\oko \to A_0$$ 
such that the following are satisfied:
\begin{enumerate}
\item $A_0$ is a CDVR with $p$ as a uniformizer and with residue field admitting a finite $p$-basis; here is the argument:
\begin{itemize}
\item  in our context, the ``$R_{0, k}$" in \cite[p. 8219]{Kim15} is simply $k_K$ (which is a field admitting a finite $p$-basis).

\item  Hence the \'etale algebra ``$\mathcal{Q}_{1, k}$" over ``$R_{0, k}$" in \cite[p. 8219]{Kim15} is nothing but a finite product of finite separable extensions of $k_K$ (hence all admitting   finite $p$-bases);  hence \emph{so are} all the inductively defined  \'etale algebras ``$\mathcal{Q}_{n+1, k}$" over ``$\mathcal{Q}_{n, k}$" for all $n \geq 1$, as well as all the ``$\mathcal{Q}_{n, k}'$" in \cite[p. 8227]{Kim15} which are quotients of ``$\mathcal{Q}_{n, k}$". 
\item  The ``$A_0$" in \cite[p. 8228]{Kim15} chosen  {in our context} is nothing but a finite product of CDVRs each with $p$ as a uniformizer and with residue field admitting a finite $p$-basis; we could simply take one factor as that already satisfies the requirement $\Spec A_0/(p)$ surjects to $\Spec k_K$ as asked in \cite[p. 8228]{Kim15}.
\end{itemize}

\item And  there is a connection operator 
\begin{equation}\label{eqc1}
\nabla:  \cn_0 \otimes_\oko A_0  \to (\cn_0 \otimes_\oko A_0) \otimes_{A_0}\wh{\Omega}_{A_0}, 
\end{equation}
  such that the induced map
\begin{equation}\label{eqc2}
  \cn_0 \otimes_\oko A_0  \onto \varphi^\ast(\gm/u\gm) \otimes_\oko A_0
\end{equation}
  is compatible with connections on both sides.  
\end{enumerate}

Let $A=A_0\otimes_{\oko}\O_K$.
Then $A[1/p]$ is a CDVF whose residue field  admits a  finite $p$-basis. Then one can define $\gs_A$ and define the relevant Breuil-Kisin modules and related categories. 
Let 
$$\gn_A:=\gs_A\otimes_\gs \gn, \quad \gm_A:=\gs_A\otimes_\gs \gm $$ 
and denote 
$$\gn_A'=\Ker(\gn_A \onto \gm_A).$$
By \cite[Prop. 10.3]{Kim15}, $\gn_A$ together with the connection \eqref{eqc1} is an object in $ \mathrm{MF}^{\mathrm{BT}}_{\gs_A}(\varphi, \nabla) $.
Because of \eqref{eqc2},  $\gn_A'$ can be equipped with an induced connection, and hence becomes an object in $ \mathrm{MF}^{\mathrm{BT}}_{\gs_A}(\varphi, \nabla) $ as well.
Since $A[1/p]$ is  a CDVF whose residue field  admits a  finite $p$-basis,   Thm. \ref{68} implies that $\gn_A' \into \gn_A$ corresponds to an isogeny $G_A' \to G_A$ of $p$-divisible groups over $A$. 
Let 
$$H_A:=\Ker(G_A' \to G_A).$$ 
Clearly, $ {\gm}^\ast(H_A)=\gm\otimes_\gs \gs_A$. 
Finally, a fpqc descent argument as in \cite[p. 8228]{Kim15} shows that $H_A$ descends to some finite flat group scheme over $\O_K$, giving the desired object mapping to $\gm$.
\end{proof}


 \bibliographystyle{alpha}

\begin{thebibliography}{BBM82}

\bibitem[AB08]{AB08}
Fabrizio Andreatta and Olivier Brinon.
\newblock Surconvergence des repr\'{e}sentations {$p$}-adiques: le cas relatif.
\newblock Number 319, pages 39--116. 2008.
\newblock Repr\'{e}sentations $p$-adiques de groupes $p$-adiques. I.
  Repr\'{e}sentations galoisiennes et $(\phi,\Gamma)$-modules.

\bibitem[ALB]{ALB}
Johannes Ansch\"{u}tz and Arthur-C\'{e}sar Le~Bras.
\newblock Prismatic {D}ieudonn\'{e} theory.
\newblock {\em preprint}.

\bibitem[And06]{And06}
Fabrizio Andreatta.
\newblock Generalized ring of norms and generalized {$(\phi,\Gamma)$}-modules.
\newblock {\em Ann. Sci. \'{E}cole Norm. Sup. (4)}, 39(4):599--647, 2006.

\bibitem[BBM82]{BBM}
Pierre Berthelot, Lawrence Breen, and William Messing.
\newblock {\em Th\'{e}orie de {D}ieudonn\'{e} cristalline. {II}}, volume 930 of
  {\em Lecture Notes in Mathematics}.
\newblock Springer-Verlag, Berlin, 1982.

\bibitem[Ber02]{Ber02}
Laurent Berger.
\newblock {Repr{\'e}sentations {$p$}-adiques et {\'e}quations
  diff{\'e}rentielles}.
\newblock {\em Invent. Math.}, 148(2):219--284, 2002.

\bibitem[Bri06]{Bri06}
Olivier Brinon.
\newblock Repr\'{e}sentations cristallines dans le cas d'un corps r\'{e}siduel
  imparfait.
\newblock {\em Ann. Inst. Fourier (Grenoble)}, 56(4):919--999, 2006.

\bibitem[Bri08]{Bri08}
Olivier Brinon.
\newblock Repr\'{e}sentations {$p$}-adiques cristallines et de de {R}ham dans
  le cas relatif.
\newblock {\em M\'{e}m. Soc. Math. Fr. (N.S.)}, (112):vi+159, 2008.

\bibitem[BT08]{BT08}
Olivier Brinon and Fabien Trihan.
\newblock Repr\'{e}sentations cristallines et {$F$}-cristaux: le cas d'un corps
  r\'{e}siduel imparfait.
\newblock {\em Rend. Semin. Mat. Univ. Padova}, 119:141--171, 2008.

\bibitem[CF00]{CF00}
Pierre Colmez and Jean-Marc Fontaine.
\newblock {Construction des repr{\'e}sentations {$p$}-adiques semi-stables}.
\newblock {\em Invent. Math.}, 140(1):1--43, 2000.

\bibitem[CL09]{CL09}
Xavier Caruso and Tong Liu.
\newblock {Quasi-semi-stable representations}.
\newblock {\em Bull. Soc. Math. France}, 137(2):185--223, 2009.

\bibitem[Gao]{Gaolp}
Hui Gao.
\newblock {Breuil-Kisin modules and integral $p$-adic Hodge theory}.
\newblock {\em preprint}.

\bibitem[Gao17]{Gao17}
Hui Gao.
\newblock {Galois lattices and strongly divisible lattices in the unipotent
  case}.
\newblock {\em J. Reine Angew. Math.}, 728:263--299, 2017.

\bibitem[GL]{GLAMJ}
Hui Gao and Tong Liu.
\newblock {Loose crystalline lifts and overconvergence of \'etale $(\varphi,
  \tau)$-modules}.
\newblock {\em to appear, Amer. J. Math}.

\bibitem[GR03]{GabberRamero}
Ofer Gabber and Lorenzo Ramero.
\newblock {\em Almost ring theory}, volume 1800 of {\em Lecture Notes in
  Mathematics}.
\newblock Springer-Verlag, Berlin, 2003.

\bibitem[Kat73]{Kat72}
Nicholas~M. Katz.
\newblock {$p$}-adic properties of modular schemes and modular forms.
\newblock In {\em Modular functions of one variable, {III} ({P}roc. {I}nternat.
  {S}ummer {S}chool, {U}niv. {A}ntwerp, {A}ntwerp, 1972)}, pages 69--190.
  Lecture Notes in Mathematics, Vol. 350, 1973.

\bibitem[Ked04]{Ked04}
Kiran~S. Kedlaya.
\newblock A {$p$}-adic local monodromy theorem.
\newblock {\em Ann. of Math. (2)}, 160(1):93--184, 2004.

\bibitem[Ked05]{Ked05}
Kiran~S. Kedlaya.
\newblock Slope filtrations revisited.
\newblock {\em Doc. Math.}, 10:447--525, 2005.

\bibitem[Kim12]{Kim12}
Wansu Kim.
\newblock The classification of {$p$}-divisible groups over 2-adic discrete
  valuation rings.
\newblock {\em Math. Res. Lett.}, 19(1):121--141, 2012.

\bibitem[Kim15]{Kim15}
Wansu Kim.
\newblock The relative {B}reuil-{K}isin classification of {$p$}-divisible
  groups and finite flat group schemes.
\newblock {\em Int. Math. Res. Not. IMRN}, (17):8152--8232, 2015.

\bibitem[Kis06]{Kis06}
Mark Kisin.
\newblock {Crystalline representations and {$F$}-crystals}.
\newblock In {\em {Algebraic geometry and number theory}}, volume 253 of {\em
  {Progr. Math.}}, pages 459--496. Birkh{\"a}user Boston, Boston, MA, 2006.

\bibitem[KL]{KL2}
Kiran Kedlaya and Ruochuan Liu.
\newblock {Relative $p$-adic {H}odge theory, {II}: imperfect period rings}.
\newblock {\em preprint}.

\bibitem[KL15]{KL15}
Kiran~S. Kedlaya and Ruochuan Liu.
\newblock Relative {$p$}-adic {H}odge theory: foundations.
\newblock {\em Ast\'erisque}, (371):239, 2015.

\bibitem[Lau14]{Lau14}
Eike Lau.
\newblock Relations between {D}ieudonn\'{e} displays and crystalline
  {D}ieudonn\'{e} theory.
\newblock {\em Algebra Number Theory}, 8(9):2201--2262, 2014.

\bibitem[Liu07]{Liu07}
Tong Liu.
\newblock {Torsion {$p$}-adic {G}alois representations and a conjecture of
  {F}ontaine}.
\newblock {\em Ann. Sci. {\'E}cole Norm. Sup. (4)}, 40(4):633--674, 2007.

\bibitem[Liu08]{Liu08}
Tong Liu.
\newblock {On lattices in semi-stable representations: a proof of a conjecture
  of {B}reuil}.
\newblock {\em Compos. Math.}, 144(1):61--88, 2008.

\bibitem[Liu13]{Liu13}
Tong Liu.
\newblock The correspondence between {B}arsotti-{T}ate groups and {K}isin
  modules when {$p=2$}.
\newblock {\em J. Th\'{e}or. Nombres Bordeaux}, 25(3):661--676, 2013.

\bibitem[LM]{LiuMoon}
Tong Liu and Yong~Suk Moon.
\newblock {Relative crystalline representations and $p$-divisible groups in the
  small ramification case}.
\newblock {\em to appear, Algebra \& Number Theory}.

\bibitem[LZ17]{LZ17}
Ruochuan Liu and Xinwen Zhu.
\newblock Rigidity and a {R}iemann-{H}ilbert correspondence for {$p$}-adic
  local systems.
\newblock {\em Invent. Math.}, 207(1):291--343, 2017.

\bibitem[Mor10]{MoritadR}
Kazuma Morita.
\newblock Hodge-{T}ate and de {R}ham representations in the imperfect residue
  field case.
\newblock {\em Ann. Sci. \'{E}c. Norm. Sup\'{e}r. (4)}, 43(2):341--356, 2010.

\bibitem[Mor14]{Moritacrys}
Kazuma Morita.
\newblock Crystalline and semi-stable representations in the imperfect residue
  field case.
\newblock {\em Asian J. Math.}, 18(1):143--157, 2014.

\bibitem[Ohk13]{Ohk13}
Shun Ohkubo.
\newblock The {$p$}-adic monodromy theorem in the imperfect residue field case.
\newblock {\em Algebra Number Theory}, 7(8):1977--2037, 2013.

\bibitem[Ray74]{Ray74}
Michel Raynaud.
\newblock Sch\'{e}mas en groupes de type {$(p,\dots, p)$}.
\newblock {\em Bull. Soc. Math. France}, 102:241--280, 1974.

\bibitem[Shi]{Shipst}
Koji Shimizu.
\newblock {A $p$-adic monodromy theorem for de {R}ham local systems}.
\newblock {\em preprint}.

\bibitem[Shi18]{Shi18}
Koji Shimizu.
\newblock Constancy of generalized {H}odge-{T}ate weights of a local system.
\newblock {\em Compos. Math.}, 154(12):2606--2642, 2018.

\bibitem[Tat67]{Tat67}
J.~T. Tate.
\newblock {$p$}-divisible groups.
\newblock In {\em Proc. {C}onf. {L}ocal {F}ields ({D}riebergen, 1966)}, pages
  158--183. Springer, Berlin, 1967.

\bibitem[Vas13]{Vas13}
Adrian Vasiu.
\newblock A motivic conjecture of {M}ilne.
\newblock {\em J. Reine Angew. Math.}, 685:181--247, 2013.

\end{thebibliography}

\end{document}